\documentclass{article}

\usepackage[a4paper,top=2cm,bottom=2cm,left=3cm,right=3cm,marginparwidth=1.75cm]{geometry}

\usepackage[T1]{fontenc}
\usepackage{lmodern}
\usepackage{pgf}
\usepackage{pifont}
\usepackage{ocg-p}	
\usepackage{amsmath,amsthm,amssymb,enumerate}
\usepackage{mathtools}
\usepackage{centernot}
\usepackage{multirow,booktabs}
\usepackage{url}
\usepackage[normalem]{ulem}
\usepackage{subcaption}

\usepackage{latexsym}
\usepackage{mismath}
\usepackage{dsfont}
\usepackage{hyperref}

\theoremstyle{plain}

\theoremstyle{definition}
\newtheorem{Def}{Definition}

\newtheorem{Lem}[Def]{Lemma}
\newtheorem{The}[Def]{Theorem}

\newtheorem{Rem}[Def]{Remark}

\newcommand{\bu}{\boldsymbol{u}}
\newcommand{\bv}{\boldsymbol{v}}
\newcommand{\demi}{\frac{1}{2}}
\newcommand{\m}{\boldsymbol{m}}
\newcommand{\yu}{\boldsymbol{u}}

\newcommand{\fx}{\nabla^+_x}
\newcommand{\fy}{\nabla^+_y}
\newcommand{\cx}{\nabla^c_x}
\newcommand{\cy}{\nabla^c_y}
\newcommand{\bx}{\nabla^-_x}
\newcommand{\by}{\nabla^-_y}

\newcommand{\Ld}{\Delta_h}

\newcommand{\dvs}{\mathbb{P}^0(\mathcal{T}_h;\R^2)} 
\newcommand{\drhos}{\mathbb{P}^0(\mathcal{T}_h)} 

\newcommand{\dd}{\,\mathrm{d}}
\newcommand{\Dhpx}{D^+_{h,x}}
\newcommand{\Dhpy}{D^+_{h,y}}

\newcommand{\Dhmx}{D^-_{h,x}}
\newcommand{\Dhmy}{D^-_{h,y}}

\newcommand{\Dhcx}{D^c_{h,x}}
\newcommand{\Dhcy}{D^c_{h,y}}
\newcommand{\normn}[1]{\lVert#1\rVert}
\newcommand{\norml}[1]{\big\lVert#1\big\rVert}

\newcommand{\BV}{\mathrm{BV}}

\newcommand{\dv}{\operatorname{div}}
\newcommand{\calE}{\mathcal E}
\newcommand{\calK}{\mathcal K}
\newcommand{\bbI}{\mathbb I}

\newcommand{\icj}{{i,j}}
\newcommand{\be}{\boldsymbol{e}}
\newcommand{\bm}{\boldsymbol{m}}
\newcommand{\bbS}{\mathbb S}
\newcommand{\bbD}{\mathbb D}
\renewcommand{\R}{\mathbb{R}}

\renewcommand{\N}{\mathbb{N}}
\newcommand{\T}{\mathbb{T}^d}
\newcommand{\dev}{\operatorname{dev}}
\newcommand{\bvarphi}{\boldsymbol{\varphi}}

\date{}

\title{Convergence analysis for a finite-volume scheme for the Euler- and Navier-Stokes-Korteweg  system via energy-variational solutions}
\author{Thomas Eiter\thanks{Department of Mathematics and Computer Science, Freie Universit\"at Berlin, Arnimallee 14, 14195 Berlin, Germany;  Weierstrass Institute for Applied Analysis and Stochastics, Anton-Wilhelm-Amo-Str.~39, 10117 Berlin, Germany;  email: \texttt{thomas.eiter@wias-berlin.de}}
\and
Jan Giesselmann\thanks{Numerical Analysis and Scientific Computing,
Department of Mathematics, Technical University of Darmstadt, Dolivostr.~15, 64293 Darmstadt, Germany, email: \texttt{Jan.Giesselmann@tu-darmstadt.de}}
\and
Robert Lasarzik\thanks{Weierstrass Institute for Applied Analysis and Stochastics, Anton-Wilhelm-Amo-Str.~39, 10117 Berlin, Germany;  email: \texttt{robert.lasarzik@wias-berlin.de}} 
\and
Philipp Öffner\thanks{Institute of Mathematics, Clausthal University of Technology, Erzstr.~1, 68378 Clausthal-Zellerfeld, Germany,
email: \texttt{philipp.oeffner@tu-clausthal.de}} 
\and 
Robert Sauerborn\thanks{Institute of Mathematics, Clausthal University of Technology, Germany, Erzstr.~1, 68378 Clausthal-Zellerfeld
email: \texttt{robert.sauerborn@tu-clausthal.de}}
}

\begin{document}
\maketitle

\begin{abstract}
We consider a structure-preserving finite-volume scheme for the Euler--Korteweg (EK) and Navier--Stokes--Korteweg (NSK) equations. 
We prove that its numerical solutions converge to energy-variational solutions of EK or NSK under mesh refinement.
Energy-variational solutions constitute a novel solution concept that has recently been introduced for hyperbolic conservation laws, including the EK system, and which we extend to the NSK model.
Our proof is based on establishing uniform  estimates following from the properties of the structure-preserving scheme, and using the stability  of the energy-variational formulation under weak convergence in the natural energy spaces. 
\end{abstract}

\medskip

\noindent
\textbf{MSC2020}:  
35D99, 
35Q35, 
65M08, 
65M12,  
76M12.  
\\
\textbf{Keywords}:
Euler--Korteweg equations, Navier--Stokes--Korteweg equations, structure-preserving scheme, finite-volume method, convergence analysis, energy-variational solutions.

\section{Introduction}
The motion of compressible fluids, which plays a central role in continuum mechanics and mathematical physics, is classically described by the Euler or Navier--Stokes equations. While these models are fundamental in fluid mechanics, they fail to capture certain physical phenomena, such as capillarity effects \cite{zbMATH07406001}. A formulation of capillarity with diffuse interfaces was first introduced by Korteweg~\cite{korteweg1901} and later derived rigorously in a thermodynamically consistent manner~\cite{DunnSerrin85}. For compressible, capillary, isothermal fluids, this leads to the study of the Euler--Korteweg (EK) system and the Navier--Stokes--Korteweg (NSK) system (\textit{cf.}~\cite{Bresch03}), which we consider with periodic boundary conditions:
\begin{subequations}\label{sys:EulKor}
	\begin{align}
	\partial_t \rho + \dv (\rho \yu) &= \,0 && \text{in }\mathbb{T}^d \times (0,T)\,, \label{eq:}\\
	\partial_t(\rho \yu) + \dv \left( \rho \yu \otimes \yu  + p(\rho)\mathbb{I} \right) &=  
    \dv\bbS(\bu)
    +\kappa \rho \nabla\Delta \rho && \text{in }\mathbb{T}^d \times (0,T)\,,
    \\
    (\rho,\rho\bu)(0)&=(\rho_0,\rho_0\bu_0)
    &&\text{in }\mathbb T^d\,.
	\end{align}
\end{subequations}
Here, the unknowns are the velocity field $\yu : \mathbb{T}^d \times (0,T) \to \R^d$ and the density $\rho : \mathbb{T}^d \times (0,T) \to \R$, where $\mathbb{T}^d$ denotes the flat, $d$-dimensional torus.
We restrict ourselves to $d=2$.
The viscous stress tensor $\bbS$ is given by Newton's rheological law
\[
\bbS(\nabla\bu) = 2\mu \dev(\bbD(\bu))
 + \eta \dv \bu \bbI ,
\qquad
\bbD(\bu)=\frac{1}{2}\big(\nabla\bu + \nabla \bu^T\big)\,,
\]
where $\dev \mathbb A=\mathbb A-\frac{1}{d}(\tr\mathbb A) \bbI $ denotes 
 the deviatoric part of a symmetric matrix $\mathbb A\in\R^{d\times d}$.
The constants $\mu\geq 0$, $\eta\geq0$ and $\kappa>0$ denote the coefficients for shear viscosity, bulk viscosity and capillarity, respectively, and
$p\colon[0,\infty)\to\R$ is a prescribed pressure law given by
$p(\rho)=k\rho^\gamma$ for $k>0$ and $\gamma>1$.
Finally, $(\rho_0,\bu_0)$ denote given initial data.
We consider the cases $\mu=\eta=0$ and $\mu,\eta>0$, corresponding to the Euler--Korteweg and Navier--Stokes--Korteweg equations. 

These systems pose significant mathematical challenges. This difficulty is already reflected at the level of general hyperbolic conservation laws, where classical solutions exists only locally in time and shock formation necessitates the introduction of weak solutions. 
Several recently developed non-uniqueness results fundamentally challenge the predictive power of these equations at the level of weak (entropy) solutions. In the incompressible case without capillarity, non-uniqueness of weak solutions was obtained for the incompressible Euler equations~\cite{DeLellisSzekely09} as well as for Leray solutions to the incompressible Navier--Stokes equations~\cite{Dallas22}.  In the absence of capillarity, the theory of compressible Euler equations is marked by non-uniqueness of weak (entropy) solutions and ill-posedness in low-regularity settings~\cite{Chioda15}. 
The existence of infinitely many global-in-time weak solutions to the Euler--Korteweg--Poisson system for sufficiently smooth initial data was shown in~\cite{Dona15}.

To ensure the global existence of solutions for a large set of initial data, 
several generalized solution concepts have been proposed in the literature. Even before the aforementioned rigorous non-uniqueness results, measure-valued solutions were introduced 
for hyperbolic conservation laws~\cite{Diperna}.
The existence of measure-valued solutions was also established for the compressible Euler and Navier--Stokes equations~\cite{Meas-CompNavSto}. In the case of the Euler--Korteweg--Poisson system, dissipative measure-valued solutions were recently shown to exist in~\cite{ExMeasEulKorte}, and the NSK case has been considered in~\cite{yang2026dissipative}.
Lions~\cite{Lions} stated an essential requirement for  generalized solution concepts: 
Solutions should satisfy a weak--strong uniqueness principle, \textit{i.e.,} coincide with a classical solution emanating from the same initial datum as long as the classical solution exists. This implies that such generalized solutions constitute natural extensions of classical solutions. Weak--strong uniqueness results typically rely on the construction of a suitable relative energy inequality, which was derived for the Euler--Korteweg and Navier--Stokes--Korteweg systems in~\cite{Giesselmann17}. This tool can also be employed to derive a posteriori error estimates~\cite{Giesselmann20}.

A refined concept of measure-valued solutions, so-called dissipative weak solutions, was introduced via the barycenter of the oscillation measure~\cite{Num} associated with a measure-valued solution, which satisfies the weak formulation up to a defect measure (\textit{cf.}~\cite{Num}). 
The most recent contribution to the family of generalized solutions is the concept of energy-variational solutions. So far, they have been considered for hyperbolic conservation laws~\cite{EVExistence} and viscoelastic fluids. Most recently, an global-in-time existence result was obtained in a very general context~\cite{EiterLasarzikSliwinski_envar}, which also yields the existence of energy-variational solutions to the Euler--Korteweg system. Energy-variational solutions identify a common convexity structure in many systems of partial differential equations that allows one to pass to the limit in suitable approximations. In certain cases, the high-dimensional defect measure of dissipative solutions can be equivalently represented by a one-dimensional auxiliary variable.
In particular, the equivalence of energy-variational solutions and dissipative weak solutions was established in~\cite{EVExistence} for both the incompressible and compressible Euler equations. In the recent work~\cite{Lasarzik25}, it was observed that this appears to be a recurrent structure in nonlinear evolution equations: energy-variational solutions were shown to be stronger than varifold solutions in the case of the two-phase Navier--Stokes equations, and stronger than Young measure-valued solutions for the hyperbolic system of polyconvex elasticity.

In the context of the EK and NSK systems~\eqref{sys:EulKor}, the weak formulation heavily relies on a conservative formulation of the third order terms introducing the Korteweg stress tensor through the identity
\[
\rho \nabla \Delta \rho
= \nabla \left ( \dv ( \rho \nabla  \rho) - \frac{1}{2}| \nabla \rho|^2 \right )
- \dv (\nabla \rho \otimes \nabla\rho ) \,,
\]
which leads to the weak formulation 
\begin{align}\label{intro:weakform}
\begin{split}
        \int_{\T} \rho \psi \dd x \Big|_s^t
        - \int_s^t \int_{\T} \rho \partial _t \psi + \bm  \cdot \nabla \psi \dd x \dd \tau &=0 \,,\\
    \int_{\T}\bm \cdot \bvarphi \dd x \Big|_s^t
    - \int_s^t \int_{\T} \bm \cdot \partial_t \bvarphi
    + \frac{\bm\otimes \bm }{\rho } : \nabla \bvarphi
    + p(\rho) \dv \bvarphi
    - \mathbb{S}(\nabla \bu)
    : \nabla \bvarphi \dd x \dd \tau & \\
    -\kappa \int_s^t\int_{\T} \rho \nabla \rho \cdot \nabla \dv \bvarphi
    + \frac{1}{2}|\nabla \rho|^2 \dv \bvarphi
    + \nabla\rho \otimes \nabla \rho : \nabla \bvarphi \dd x \dd \tau &= 0 \,,
\end{split}
\end{align}
for all $ s,t\in[0,T]$ and all test functions
$\psi  \in  C^1([0,T];C^1(\T;\R))$ and
$\bvarphi \in C^1([0,T]; C^{2}(\T;\R^d))$.

Here, $\bm=\rho\bu$.
Formally, this system admits the energy-dissipation inequality 
\begin{align}\label{intro:en}
  \mathcal{E}(\rho,\bm) \Big|_s^t
  +\int_s^t \int_{\mathbb{T}^d}
\left(
2\mu |\dev(\mathbb{D}(\bu))|^2
+\eta (\operatorname{div} \bu)^2
\right)
\dd x \dd\tau \leq 0
\end{align}
for the total energy 
\begin{equation}
\label{eq:energy}
  \mathcal{E}(\rho,\bm)
  :=\int_{\T} \tilde{\eta}(\rho(x),m(x))+\frac{\kappa |\nabla \rho(x)|^2}{2}\dd x,
\end{equation}
where $\tilde{\eta}: \R \times \R^d \rightarrow [0,\infty]$ is given by
\[
    \tilde{\eta}(\rho, \bm) = \begin{cases}
    \frac{|\bm|^2}{2\rho}+P(\rho) \quad &\text{if} \quad \rho>0, \\
    0 \quad &\text{if}\quad (\rho,m)=(0,0),\\
    \infty \quad &\text{else},
       \end{cases}
\]
with $P$ defined by the identity $\rho P'(\rho)=p(\rho)+P(\rho)$.

To define energy-variational solutions, 
we now introduce an auxiliary variable $E$ as an upper bound for the energy, $E\geq \mathcal{E}(\rho,\bm)$. Subtracting~\eqref{intro:weakform} from~\eqref{intro:en} and adding the energy defect $\mathcal{E}(\rho,\bm)-E$, multiplied by a suitable regularity weight $\mathcal{K}: C^2(\T;\R^2)\to [0,\infty)$, yields 
\begin{equation}\label{intro:envar.ek}
\begin{aligned}
    \left[E- \int_{\mathbb T^2} \rho  \psi + \bm \cdot\bvarphi \dd{x} \right]\bigg|_s^t
    &+\int_s^t \int_{\mathbb T^2} \rho \, \partial_t \psi+ \bm \cdot\nabla \psi
    + \bm \cdot \partial_t \bvarphi
    + \left(\frac{\bm {\otimes} \bm}{\rho}+ p(\rho) \mathbb{I} \right){:}\nabla \bvarphi \dd{x} \dd{\tau} \\&+\int_s^t\int_{\mathbb T^2}  \bbS(\bu):\bbD(\bu) - \bbS(\bu):\bbD(\bvarphi)\dd{x}\dd{\tau}\\
    &+\int_s^t\int_{\mathbb T^2} \kappa \rho\nabla \rho \cdot \nabla (\dv \bvarphi)
    + \frac{\kappa}{2} |\nabla \rho|^2\dv \bvarphi
    + \kappa\nabla \rho \otimes \nabla \rho:\nabla \bvarphi \dd{x}\dd{\tau}\\
    &+ \int_s^t \mathcal{K}(\bvarphi)(\mathcal{E}(\rho,\bm)-E) \dd{\tau}\leq 0
\end{aligned}
\end{equation}
for all test functions $\psi\in C^1([0,T]; C^1({\mathbb T^2}))$ and
$\bvarphi \in C^1([0,T];C^2({\mathbb T^2};\R^d))$, and for almost every $s<t$.
With an appropriate choice of the regularity weight $\mathcal K$, the left-hand side is lower semicontinuous with respect to the variables $(\rho, \bm, E)$, rendering the formulation stable under convergence in the natural energy spaces. Consequently, this formulation is stable under weak convergence without the need to adopt measure-valued spaces and is therefore amenable to numerical analysis.

Thus, generalized solution frameworks are not only of analytical interests but also play a crucial role in the convergence analysis of (high-order) structure-preserving numerical schemes. 
An important concept in this context is that of dissipative weak solutions, together with corresponding convergence results. In \cite{foundations}, classical finite-volume schemes with Lax--Friedrichs and Rusanov fluxes are shown to converge for the Euler equations, and in \cite{yuan2023} convergence of the Godunov method has been shown for the complete Euler equations. The book \cite{Num} provides a survey of convergence results for various finite-volume-type schemes applied to both the barotropic and the complete Euler as well as the Navier--Stokes equations. These works form the basis for more recent convergence analyses of higher-order schemes. In particular, convergence results have been established for residual-distribution schemes in \cite{abgrall2022convergence}, for flux-corrected finite-element methods in \cite{zbMATH08167643}, and for discontinuous Galerkin methods in \cite{oeffner2023}. More recently, \cite{arXiv:2601.17452} proved convergence of a hyperbolic thermodynamically compatible finite-volume scheme.

In a different context, namely, the Ericksen--Leslie equations modeling liquid crystals, energy-variational solutions were shown to be beneficial for the numerical analysis of structure-inheriting schemes. In particular, in~\cite{LasarzikReiter23}, convergence of a fully discrete finite-element scheme was proved in this solvability framework. 
Thus, the energy-variational concept has proved to be well suited for identifying limits of numerical schemes. In the present work, we employ this framework to show convergence of solutions to a recently introduced finite-difference scheme for the EK and NSK  systems~\cite{giesselmann2026structurepreservingfinitevolume}, which 
is structure preserving in the sense that total mass and momentum are conserved, and the natural energy-dissipation inequality is satisfied. 
We establish convergence of the numerical solutions toward an energy-variational solution. Working within the energy-variational framework allows us to remain in standard energy spaces and to avoid the introduction of Young measures, which entail additional complexity. Instead of employing multiple defect measures to describe limit terms that cannot be identified directly, we rely solely on the real-valued energy defect and the weak sequential closedness of the energy-variational formulation. 

To conclude this introduction, let us further give a brief overview of numerical schemes for the EK and NSK systems. Existing approaches for the NSK system include energy-consistent discontinuous Galerkin (DG) methods \cite{zbMATH08071323, zbMATH06342342}, which violate strict conservation of momentum and local DG schemes \cite{zbMATH06660648}.  For the EK system, entropy-stable finite-difference schemes were  developed in  \cite{zbMATH06579844}. Let us also mention schemes based on parabolic  \cite{keim2023} or hyperbolic \cite{zbMATH07599597} relaxation of the NSK system.
Whether the convergence analysis developed here can be extended to those schemes is an interesting question but beyond the scope of this paper. 
Finally, let us mention that, to the best of our knowledge, the paper at hand is the first to assert convergence of a numerical scheme for the EK and NSK equations without assuming existence of a strong solution. 

The paper is organized as follows: 
In Section~\ref{sec:scheme},
we recall the finite-volume scheme introduced in~\cite{giesselmann2026structurepreservingfinitevolume},
state its fundamental properties,
and construct the corresponding piecewise constant approximate solutions.
In Section~\ref{sec:convergence}, 
we study their convergence.
From \textit{a priori} bounds,
we derive convergence of a subsequence
and identify the limit in terms of an energy-variational solution to the Euler--Korteweg or the Navier--Stokes Korteweg system.

\section{The finite-volume method}
\label{sec:scheme}

In this section we state the semidiscrete finite-volume method for the two-dimensional Euler--Korteweg and Navier--Stokes--Korteweg system~\eqref{sys:EulKor}
that was introduced in \cite{giesselmann2026structurepreservingfinitevolume},
and we collect some relevant properties.

\subsection{The semidiscrete scheme}
We realize the torus $\mathbb T^2$ in terms of 
the rectangular computational domain $\Omega = (0,L_x)\times(0,L_y)$ with $L_x,L_y>0$.
We divide $\Omega$ into 
Cartesian grid with cell centers
$
x_i = i  h_x, \
y_j = j  h_y,
$
and define the cell interfaces
$
x_{i\pm \frac{1}{2}} = x_i \pm \frac{1}{2} h_x, \
y_{j\pm \frac{1}{2}} = y_j \pm \frac{1}{2} h_y
$
where $h_x=L_x/M$ and $h_y=L_y/N$ are the corresponding mesh sizes
for $M, N \in \mathbb{N}$. 
The control volumes (cells) are given by
$
\Omega_{i,j}
= (x_{i-\frac{1}{2}}, x_{i+\frac{1}{2}}]
\times
(y_{j-\frac{1}{2}}, y_{j+\frac{1}{2}}].
$

For piecewise constant $\phi^h$ on this grid, 
that is, $\phi|_{\Omega_\icj}=\phi_\icj$,
we define the discrete difference operators
\begin{align*}
	\nabla^\pm\phi_{i,j} &\coloneq \begin{pmatrix}
		\nabla^\pm_x\phi_{i,j}  \\
		\nabla^\pm_y\phi_{i,j}\end{pmatrix} \coloneqq \pm \begin{pmatrix}
		\frac{\phi_{i\pm 1,j} - \phi_{i,j}}{h_x}  \\
		\frac{\phi_{i,j \pm1} - \phi_{i,j}}{h_y}
	\end{pmatrix},
    \qquad
    \nabla^c\phi_{i,j} \coloneqq		
	\begin{pmatrix}
		\cx\phi_{i,j}  \\
		\cy\phi_{i,j}
	\end{pmatrix} \coloneqq
	\begin{pmatrix}
		\frac{\phi_{i+1,j} - \phi_{i-1,j}}{2h_x}  \\
		\frac{\phi_{i,j+1} - \phi_{i,j-1}}{2h_y}
	\end{pmatrix},
    \\
    \Ld\phi_{i,j} &\coloneqq \frac{\phi_{i+1,j}-2\phi_{i,j}+\phi_{i-1,j}}{h_x^2}+\frac{\phi_{i,j+1}-2\phi_{i,j}+\phi_{i,j-1}}{h_y^2}.
\end{align*}
Here, we set $\phi_{i+kM, j+lN} \coloneqq \phi_{i,j}$ for $k,l \in \mathbb{Z}$,
which realizes the periodic boundary conditions. 

We approximate solutions 
$(\rho, \bu) = (\rho, u, v)$ to~\eqref{sys:EulKor}
by cell-wise constant functions
$(\rho^h,\bu^h)=(\rho^h,u^h,v^h)\colon\Omega\times(0,T)\to\R^3$
defined by 
\[
\rho^h(x,t)\coloneqq\sum_{i=1}^{M}\sum_{j=1}^{N} \rho_\icj(t) 1_{\Omega_\icj}(x), \qquad
\bu^h(x,t)\coloneqq\sum_{i=1}^{M}\sum_{j=1}^{N} \bu_\icj(t) 1_{\Omega_\icj}(x).
\]
Here, we set $h \coloneqq \max\lbrace h_x, h_y\rbrace$.
In what follows, we always assume $h_x=\theta_x h$ and $h_y=\theta_y h$ for some fixed ratios $\theta_x,\theta_y>0$.
The grid-wise values
$(\rho_{\icj},\bu_{\icj})=(\rho_{\icj},u_{\icj},v_{\icj})$, 
$i=1,\dots,M$, $j=1,\dots,N$,
are determined
as solutions to the system of ordinary differential equations
\begin{subequations}\label{method:semi2d}
	\begin{align}
		&\frac{\di}{\di t}  \rho_{i,j}
        +\cx(\rho u)_{i,j}+ \cy(\rho v)_{i,j} - h\lambda \Ld\rho_{i,j}=0; \label{semi2drho}
        \\	
		&\frac{\di}{\di t}  (\rho u)_{i,j}  + \cx(\rho u^2)_{i,j}+ \cy(\rho u v)_{i,j} + \cx p(\rho_{i,j})  - h\lambda \Ld(\rho u)_{i,j}\notag\\
        &\quad -\mu (\bx\fx u_{i,j} + \by \fy u_{i,j} + \by \fx v_{i,j} - \bx \fy v_{i,j})\notag \\
        &\quad - \eta (\bx \fx u_{i,j} + \bx \fy v_{i,j}) \notag \\ 
		&\quad- \kappa \Bigg[ \bx \frac{\rho_{i,j}\Ld\rho_{i+1,j} + \rho_{i+1,j}\Ld\rho_{i,j}}{2}-\demi \bx \left( \fx\rho_{i,j}\right)^2\ \label{semi2du}\\
		&\qquad\qquad + \demi \bx\left(  \by\rho_{i+1,j}\by \rho_{i,j}\right)   - \by \left( \cx\rho_{i,j}\, \fy \rho_{i,j}\right) \Bigg]=0; \notag 
        \\
		&\frac{\di}{\di t}  (\rho v)_{i,j}   
        +\cy(\rho v^2)_{i,j}+ \cx(\rho u v)_{i,j} + \cy p(\rho_{i,j})  - h\lambda \Ld(\rho v)_{i,j} \notag\\
        & \quad - \mu (\bx \fy u_{i,j} + \by \fy v_{i,j} + \bx \fx v_{i,j} - \by \fx u_{i,j})
        \notag\\
        & \quad - \eta ( \by \fx u_{i,j} + \by \fy v_{i,j})
        \notag \\
		&\quad- \kappa \Bigg[ \by \frac{\rho_{i,j}\Ld\rho_{i,j+1} + \rho_{i,j+1}\Ld\rho_{i,j}}{2}-\demi \by \left( \fy\rho_{i,j}\right)^2 \label{semi2dv} \\
		&\qquad\qquad + \demi \by\left(  \bx\rho_{i,j+1}\,\bx \rho_{i,j}\right)   - \bx \left( \cy\rho_{i,j}\, \fx \rho_{i,j}\right) \Bigg]=0;  \notag
	\end{align}
together with the initial conditions
\begin{equation}
\label{eq:initialdata.cellwise}
\rho_{i,j}(0)=\rho_{i,j}^0 := \frac{1}{|\Omega_{i,j}|}\int_{\Omega_{i,j}}\rho_0 \di x,
\qquad
(\rho\bu)_{i,j}(0)=\bm_{i,j}^0 := \frac{1}{|\Omega_{i,j}|}\int_{\Omega_{i,j}}\rho_0\bu_0 \di x.
\end{equation}
\end{subequations}
Here, $\lambda\geq 0$ is a numerical dissipation parameter,
a suitable choice of which is specified later, see Theorem~\ref{2Ddissipation} below.
The initial-value problem~\eqref{method:semi2d} corresponds to the semidiscrete finite-volume method for the two-dimensional Euler-- and Navier--Stokes--Korteweg system \eqref{sys:EulKor} introduced in~\cite{giesselmann2026structurepreservingfinitevolume}.

To express the scheme~\eqref{method:semi2d} in terms of $(\rho^h,\bu^h)$, 
for functions $f\colon\Omega\to\R$ we define
\[
\begin{aligned}
    D^\pm_{h,x} f(x,y)&\coloneqq \pm \frac{f(x\pm h_x,y)-f(x,y)}{h_x},
    &
    \Dhcx f(x,y)&\coloneqq \frac{f(x+h_x,y) - f(x-h_x,y)}{2h_x},
    \\
    D^\pm_{h,y} f(x,y)&\coloneqq \pm \frac{f(x,y\pm h_y)-f(x,y)}{h_y},
    &
    \Dhcy f(x,y)&\coloneqq \frac{f(x,y+h_y) - f(x,y-h_y)}{2h_y},
\end{aligned}
\]
and
\[
    \nabla^\pm_{h} f\coloneqq \big( D^\pm_{h,x} f,  D^\pm_{h,y} f\big),
    \qquad
    \nabla^c_{h} f\coloneqq \big( D^c_{h,x} f,  D^c_{h,y} f\big),
    \qquad
    D^2_h f := \Dhmx\Dhpx f + \Dhmy\Dhpy f.
\]
Then~\eqref{method:semi2d} is satisfied
if and only if the space-discrete functions $(\rho^h,\bu^h)$ are subject to
\begin{subequations}
\label{eq:EK.approx}
\begin{align}
	&\partial_t\rho^h + \Dhcx(\rho^h u^h) +\Dhcy(\rho^h v^h) -\lambda h D_h^2\rho^h =0, 
    \label{eq:EK.approx.cont}
    \\	
    \begin{split}
	&\partial_t(\rho^h u^h) +  \Dhcx(\rho^h (u^h)^2) + \Dhcy(\rho^h u^hv^h)+ \Dhcx p(\rho^h) - \lambda h D_h^2 (\rho^h u^h) 
    \\
    &\quad -\mu (\Dhmx\Dhpx u^h + \Dhmy\Dhpy u^h + \Dhmy\Dhpx v^h - \Dhmx\Dhpy v^h) 
    \\
    &\quad- \eta (\Dhmx \Dhpx u^h + \Dhmx\Dhpy v^h)  
    \\
	&\quad- \kappa \Big[ \Dhmx\frac{\rho^h D_h^2\rho^h(\cdot+h_x\be_x)+\rho^h(\cdot+h_x\be_x) D_h^2\rho^h}{2} - \demi \Dhmx(\Dhpx\rho^h)^2
    \\
    &\qquad\qquad
    +\demi \Dhmx\big(  \Dhmy\rho^h\Dhmy \rho^h(\cdot+h_x\be_x)\big)   - \Dhmy \big( \Dhcx\rho^h\, \Dhpy \rho^h\big)\Big] =0,
    \end{split}
    \label{eq:EK.approx.mom.u}
    \\
    \begin{split}
	&\partial_t(\rho^h v^h) +  \Dhcy(\rho^h (v^h)^2) + \Dhcx(\rho^h u^hv^h)+ \Dhcy p(\rho^h) - \lambda h D_h^2 (\rho^h v^h) - \mu  D_h^2 v^h  
    \\
    & \quad - \mu (\Dhmx\Dhpy u^h + \Dhmy\Dhpy v^h + \Dhmx\Dhpx v^h- \Dhmy\Dhpx u^h)
    \\
    & \quad - \eta ( \Dhmy\Dhpx u^h + \Dhmy\Dhpy v^h)
    \\
	&\quad- \kappa \Big[ \Dhmy\frac{\rho^h D_h^2\rho^h(\cdot+h_y\be_y)+\rho^h(\cdot+h_y\be_y) D_h^2\rho^h}{2} - \demi \Dhmy(\Dhpy\rho^h)^2
    \\
    &\qquad\qquad
    +\demi \Dhmy\big(  \Dhmx\rho^h\Dhmx \rho^h(\cdot+h_y\be_y)\big)   - \Dhmx \big( \Dhcy\rho^h\, \Dhpx \rho^h\big)\Big] =0,
    \end{split}
    \label{eq:EK.approx.mom.v}
\end{align}
in $(0,T)\times\Omega$
and satisfy the initial conditions 
\begin{equation}\label{eq:definitial}
(\rho^h,\rho^h\yu^h)(0)=(\rho^h_0,\bm^h_0),
\end{equation}
\end{subequations}
where the spatially discretized initial conditions 
$(\rho^h_0,\bm^h_0)$
are given by
\[
\rho_0^h(x)\coloneqq\sum_{i=1}^M\sum_{j=1}^N 1_{\Omega_\icj}(x)\rho^0_\icj, \qquad
\bm_0^h(x)\coloneqq\sum_{i=1}^M\sum_{j=1}^N 1_{\Omega_\icj}(x)\bm_\icj^0,
\]
with $\rho_\icj^0$ and $\mathbf{m}_\icj^0$ given by~\eqref{eq:initialdata.cellwise}.

Note that \eqref{eq:EK.approx.cont}--\eqref{eq:EK.approx.mom.v} constitutes a finite-dimensional system of ordinary differential equations whose right-hand side is Lipschitz continuous as long as the densities in all cells remain positive. In particular, for initial data with positive density, a solution to \eqref{eq:EK.approx} exists locally in time. However, establishing global existence is beyond the scope of the work at hand. Theoretically, it might  even happen that the existence interval shrinks to zero as  $h$ decreases. A key assumption of our convergence analysis is that there is a time interval of existence that is uniformly in $h$.
This is the same type of assumption as in
other convergence studies, see~\cite{Feireisl2004,zbMATH08167643, Num} for instance.

\subsection{Properties}
We next collect some fundamental properties of the scheme. 
We consider initial data $(\rho_0, \m_0)$ with $\rho_0 \geq \underline{\rho} >0$ and 
finite total energy, that is, 
$\mathcal{E}(\rho_0,\bm_0) < \infty$
with $\mathcal E$ defined in~\eqref{eq:energy}.
 A discrete equivalent to the total energy $\mathcal{E}$ is given by 
	\begin{align}\label{eq:energy2d}
	E^h= \frac{|\Omega|}{MN} \sum_{i=1}^M\sum_{j=1}^N \left(  \tilde{\eta}(\rho_{i,j}, \rho_{i,j} \yu_{i,j}) + \frac{\kappa}{2} \left| \nabla^+\rho_{i,j}\right| ^2\right) \, ,
\end{align} 
which can also be expressed by an integral functional as
\begin{equation}\label{eq:energy.discrete}
E^h=\int_{\mathbb T^2} e^h(x)\dd x, 
\qquad
e^h:=\tilde\eta(\rho^h,\rho^h\bu^h)
+\frac{\kappa}{2}|\nabla_h^+\rho^h|^2.
\end{equation}

In \cite{giesselmann2026structurepreservingfinitevolume}, the conservation of mass and momentum as well as the dissipation of the discrete energy \eqref{eq:energy2d} were proven for an analogous scheme that discretizes an equation that coincides with \eqref{sys:EulKor} up to the fact that the Newtonian stress tensor is replaced with the simpler expression $\mu \nabla \bu$.

\begin{The}\label{2Dconservation}
	The semidiscrete method \eqref{method:semi2d} conserves total mass and total momentum,
    which are given by
    \[
    \int_{\mathbb T^2}\rho^h\dd x
    =\frac{|\Omega|}{MN} \sum_{i=1}^{M}\sum_{j=1}^{N} \rho_{i,j},
    \qquad
    \int_{\mathbb T^2}\rho^h\bu^h\dd x
    =\frac{|\Omega|}{MN}\sum_{i=1}^{M}\sum_{j=1}^{N} \rho_{i,j}\yu_{i,j}
    \]
    More precisely, every solution $(\rho^h,\bu^h)$ 
    to~\eqref{eq:EK.approx} satisfies
    \[
    \int_{\mathbb T^2}\rho^h(t)\dd x
    =\int_{\mathbb T^2}\rho^h_0\dd x,
    \qquad
    \int_{\mathbb T^2}\rho^h(t)\bu^h(t)\dd x
    =\int_{\mathbb T^2}\bm_0^h\dd x
    \]
    for $t\in [0,T^*)$, where $T^ *$ is the maximal existence time of the solution to \eqref{method:semi2d}.
\end{The}
\begin{proof}
    The proof is completely analogous to the proof of \cite[Theorem 2]{giesselmann2026structurepreservingfinitevolume}.
\end{proof}
With a particular choice of the numerical dissipation parameter $\lambda$, 
we can ensure that the semidiscrete method dissipates the total discrete energy \eqref{eq:energy2d}.

\begin{The}\label{2Ddissipation}
    Let 
    \begin{equation}
    \label{eq:lambda}
        \lambda = \demi \max_{i,j} \left\lbrace |\yu_{i,j}| + \sqrt{p'(\rho_{i,j})}\right\rbrace.
    \end{equation}
    Then every solution $(\rho_\icj,\bu_\icj)_\icj$ 
    to the semidiscrete method \eqref{method:semi2d} satisfies the inequality 
	\begin{equation}\label{eq:edi.EK.disc}
        \begin{aligned}
		\frac{\di}{\di t}E^h[\rho,\rho \yu] 
        &+ 2\mu \frac{|\Omega|}{MN}\sum_{i=1}^{M}\sum_{j=1}^{N}  \left|\dev(\bbD^+(\bu_\icj))\right|^2
        \\
        &+ \eta \frac{|\Omega|}{MN}\sum_{i=1}^{M}\sum_{j=1}^{N}  |\dv^+\bu_\icj|^2
        + \kappa \frac{|\Omega|}{MN}\sum_{i=1}^{M}\sum_{j=1}^{N} \lambda h (\Ld\rho_{i,j})^2\leq 0
        \end{aligned}
	\end{equation} 
    in its interval of existence $(0,T^*)$,
    where
    \[
    \bbD^+(\bu_\icj):=\frac{1}{2}\big(\nabla^+ \bu_{i,j} + (\nabla^+ \bu_{i,j} )^T\big),
    \qquad
    \dv^+\bu_\icj:=\fx u_{i,j} + \fy v_{i,j}.
    \]
\end{The}

\begin{Rem}\label{rem:lambda1}
    The choice of $\lambda$ in \eqref{eq:lambda} is actually the smallest choice of $\lambda$ for which we can prove energy dissipation. 
\end{Rem}
\begin{proof}
    The proof is analogous to the proof of \cite[Theorem 3]{giesselmann2026structurepreservingfinitevolume} with the exception of the viscous terms. The shear viscosity terms with coefficient $\mu$ can be treated as follows.
   For any pair of grid functions $\bvarphi^h=(\phi^h, \psi^h)$ we have via summation by parts that
   \begin{align*}
  & \sum_{i=1}^M\sum_{j=1}^N  (\bx\fx u_{i,j} + \by \fy u_{i,j} + \by \fx v_{i,j} - \bx \fy v_{i,j})\phi_{i,j}\\
       &+ \sum_{i=1}^M\sum_{j=1}^N (\bx \fy u_{i,j} + \by \fy v_{i,j} + \bx \fx v_{i,j} - \by \fx u_{i,j})\psi_{i,j}  \\
    &   =-  \sum_{i=1}^M\sum_{j=1}^N  \begin{pmatrix} 
    \fx u_{i,j} - \fy v_{i,j} & \fy u_{i,j} + \fx v_{i,j}\\
    \fy u_{i,j} + \fx v_{i,j} & \fy v_{i,j} - \fx u_{i,j}
    \end{pmatrix}: \begin{pmatrix} 
    \fx \phi_{i,j} & \fy \phi_{i,j} \\
     \fx \psi_{i,j} & \fy \psi_{i,j}
    \end{pmatrix}
    \\
    &=-  \sum_{i=1}^M\sum_{j=1}^N \dev\big(\bbD^+(\bu_\icj)\big):\nabla^+\bvarphi_\icj.
   \end{align*} 
   In the proof of the energy inequality, we use $\phi^h= u^h$ and $\psi^h = v^h$.
   The terms containing $\eta$ can be treated analogously.
\end{proof}

\section{Convergence towards energy-variational solutions}
\label{sec:convergence}

As the main result of this article, we show that the solutions to the 
proposed semi-discrete finite-volume scheme
converge in a suitable topology as $h\to0$,
and that the limit can be identified with an energy-variational solution.
We focus on the case of a barotropic pressure $p(\rho)=k\rho^\gamma$ for some $k>0$ and $\gamma>1$,
which leads to $P(\rho)=\frac{k\rho^\gamma}{\gamma-1}$ and 
to the convex total energy 
\begin{equation}\label{eq:energy.ek}
\mathcal{E}(\rho,\bm)=\int_{\mathbb T^2} \tilde{\eta}(\rho(x),\bm(x))+\frac{\kappa}{2}|\nabla \rho(x)|^2\dd{x},
\end{equation}
where 
$\tilde{\eta}: \R \times \R^2 \rightarrow [0,\infty]$ is given by 
\[
    \tilde{\eta}(\rho, \bm) = \begin{cases}
    \frac{|\bm|^2}{2\rho}+\frac{k\rho^\gamma}{\gamma-1} \quad &\text{if} \quad \rho>0, \\
    0 \quad &\text{if}\quad (\rho,m)=(0,0),\\
    \infty \quad &\text{else}.
       \end{cases}
\]
Moreover, we choose the so-called 
regularity weight $\mathcal K\colon C^2({\mathbb T^2};\R^2)\to[0,\infty)$
with 
\begin{equation}
\label{eq:K}
\calK(\bvarphi)
= 2\|(\nabla \bvarphi)_{\mathrm{sym},-}\|_{L^\infty({\mathbb T^2})}+\max\{(\gamma-1,1)\}\|(\dv\bvarphi)_{-}\|_{L^\infty({\mathbb T^2})}.
\end{equation}
Here $(A)_{\mathrm{sym},-}$ denotes the negative semi-definite part 
of the symmetric matrix $(A)_{\mathrm{sym}}=\frac{1}{2}(A+A^T)$
for $A\in\R^{2\times2}$.

\begin{Def}\label{def:envar.ek}
Let $(\rho_0,\bm_0)\in L^1({\mathbb T^2})\times L^1({\mathbb T^2};\R^2)$ such that $\calE(\rho_0,\bm_0)<\infty$.
A triple 
\begin{equation}
    \label{eq:fct.class}
    (\rho,\bm,E)\in L^\infty(0,T;H^1(\Omega) \cap L^\gamma(\Omega)) \times L^\infty(0,T;L^\frac{2\gamma}{\gamma +1}(\Omega;\R^2)) \times \mathrm{BV}([0,T])
\end{equation}
is called an \emph{energy-variational solution to the Euler--Korteweg system (if $\mu=\eta=0$) or the Navier--Stokes--Korteweg system (if $\mu,\eta>0$)} if $\mathcal{E}(\rho(t),\bm(t)) \leq E(t) $ for a.e.~$t \in (0,T)$ and  if 
\begin{equation}
\label{eq:envar.ek}
\begin{aligned}
    \left[E- \int_{\mathbb T^2} \rho  \psi + \bm \cdot\bvarphi \dd{x} \right]\bigg|_s^t &+\int_s^t \int_{\mathbb T^2} \rho \, \partial_t \psi+ \bm \cdot\nabla \psi + \bm \cdot \partial_t \bvarphi + \left(\frac{\bm {\otimes} \bm}{\rho}+ p(\rho) \mathbb{I} \right){:}\nabla \bvarphi \dd{x} \dd{\tau} \\
    &+\int_s^t\int_{\mathbb T^2}  \bbS(\bu):\bbD(\bu) - \bbS(\bu):\bbD(\bvarphi)\dd{x}\dd{\tau}\\
  &+\int_s^t\int_{\mathbb T^2} \kappa \rho\nabla \rho \cdot \nabla (\dv
       \bvarphi) + 
 \frac{\kappa}{2} |\nabla \rho|^2\dv \bvarphi  + \kappa\nabla \rho \otimes \nabla \rho:\nabla \bvarphi \dd{x}\dd{\tau}\\
     &+ \int_s^t \mathcal{K}(\bvarphi)(\mathcal{E}(\rho,\bm)-E) \dd{\tau}\leq 0
\end{aligned}
\end{equation} 
holds for all test functions $\psi\in C^1([0,T]; C^1({\mathbb T^2}))$
and $\bvarphi \in C^1([0,T];C^2({\mathbb T^2};\R^2))$,
and for a.a.~$s,t\in[0,T]$, $s<t$, including $s=0$ with $\rho(0)=\rho_0,\, \bm(0)=\bm_0$. In the case that $\mu, \eta >0$, 
 it additionally  holds that $\bm=\rho\bu$ and $\bu\in L^2(0,T;H^1({\mathbb T^2};\R^2))$.
\end{Def}

\begin{Rem}
One can split the variational inequality~\eqref{eq:envar.ek}
into the three parts
\begin{align}
    E(t)+\int_s^t\int_{\mathbb T^2}\bbS(\bu):\bbD(\bu)\dd x \dd\tau&\leq E(s),
    \label{eq:envar.ek.enineq}
    \\
    \int_0^T\int_{\mathbb T^2} \rho \, \partial_t \psi+ \bm \cdot\nabla \psi  \dd{x} \dd{\tau} &=-\int_{\mathbb T^2}\rho_0\psi(0)\dd x,
    \label{eq:envar.ek.cont}
    \\
    \begin{split}
    \int_0^T\int_{\mathbb T^2} \bm \cdot\partial_t \bvarphi + \left(\frac{\bm \otimes \bm}{\rho}+ p(\rho) \mathbb{I} \right):\nabla \bvarphi
    -\bbS(\bu)&:\bbD(\bvarphi)
     \dd{x}\dd{\tau}
       \\
    +\int_0^T\int_{\mathbb T^2} \kappa\rho\nabla \rho \cdot \nabla (\dv
       \bvarphi) + \frac{\kappa}{2}  |\nabla \rho|^2\dv \bvarphi  &+ \kappa\nabla \rho \otimes \nabla \rho{:}\nabla \bvarphi \dd{x}\dd{\tau}
 \\
    \leq -\int_{\mathbb T^2}\bm_0\cdot\bvarphi(0)\dd x 
    &+\int_0^T \mathcal{K}(\bvarphi)(E-\mathcal{E}(\rho,\bm)) \dd{\tau}
    \end{split}
    \label{eq:envar.ek.mom}
\end{align}
for a.a.~$0<s<t<T$, all $\psi\in C^1_c([0,T); C^1({\mathbb T^2}))$
and $\bvarphi \in C^1_c([0,T);C^2({\mathbb T^2};\R^2))$.
Here,~\eqref{eq:envar.ek.enineq} corresponds to the energy-dissipation inequality,
while~\eqref{eq:envar.ek.cont} and~\eqref{eq:envar.ek.mom} are weak formulations of the continuity equation and an adaption of the momentum equation, respectively.
Indeed, their collection is equivalent to~\eqref{eq:envar.ek},
which can be seen by classical variational arguments
since $\calK$ is positively $1$-homogeneous.
\end{Rem}

\begin{Rem}
The variational inequality~\eqref{eq:envar.ek.mom} can be equivalently reformulated as a variational equality with additional defect measures.
To define these, we introduce the space  $L^\infty_{w^*}(0,T;\mathbb V^*)$, which denotes the Bochner space of weakly* measurable functions with values in the dual of a Banach space $\mathbb V$. This is the correct dual to $ L^1(0,T;\mathbb V)$ in the case that $ \mathbb V^*$ is not separable,
     compare~\cite[Theorem 6.14]{pedregal1997parametrizedmeasures}.
Inequality~\eqref{eq:envar.ek.mom} holds for a.a.~$0<s<t<T$
and all $\bvarphi \in C^1_c([0,T);C^2({\mathbb T^2};\R^2))$ if and only if there exist two defect measures 
$\mathfrak R_1 \in L^\infty_{w^*} (0,T;\mathcal{M}( {\mathbb T^2} ; \R^{2\times 2}_{\mathrm{sym},+}))$ and  $\mathfrak R_2\in L^\infty_{w^*} (0,T;\mathcal{M}( {\mathbb T^2} ; \R_+)) $, 
such that 
\begin{align*}
      &\int_0^T\int_{\mathbb T^2} \bm \cdot\partial_t \bvarphi + \left(\frac{\bm \otimes \bm}{\rho}+ p(\rho) \mathbb{I} \right){:}\nabla \bvarphi
    -\bbS(\bu):\bbD(\bvarphi)
     \dd{x}\dd{\tau}
       \\
    &\qquad
    +\int_0^T\int_{\mathbb T^2} \kappa\rho\nabla \rho \cdot \nabla (\dv
       \bvarphi) + \frac{\kappa}{2}  |\nabla \rho|^2\dv \bvarphi  + \kappa\nabla \rho \otimes \nabla \rho{:}\nabla \bvarphi \dd{x}\dd{\tau}
 \\
    &\qquad\qquad
    = -\int_{\mathbb T^2}\bm_0\cdot\bvarphi(0)\dd x 
    -\int_0^T \Big(\int_{{{\mathbb T^2}}} \mathbb D( \bvarphi) :\dd \mathfrak R_1 + \int_{{{\mathbb T^2}}} (\dv \bvarphi) \,\dd \mathfrak R_2\Big) \dd{\tau}
\end{align*}
holds for all $\bvarphi \in C^1_c([0,T);C^2({\mathbb T^2};\R^2))$ and the measures are additionally controlled by the energy defect through
\begin{equation}
     \int_{{{\mathbb T^2}}} \dd \max\left\{\frac{1}{2}\tr ( \mathfrak R_1); | \mathfrak R_1|_2\right\}+ \frac{1}{\gamma-1} 
    \int_{{\mathbb T^2}}\dd \mathfrak R_2 \leq E -  \mathcal{E}(\rho,m)\label{eq:EulKorEn}\,
\end{equation}
a.e.~$t\in(0,T)$. 
The term $\max\left\{\frac{1}{2}\tr ( \mathfrak R_1); | \mathfrak R_1|_2\right\}$ should be interpreted as the total variation measure of $ \mathfrak R_1$ with respect to the norm induced by $|A|_m=\max\left\{\frac{1}{2}| A|_{\tr}; | A|_2\right\}$ for a symmetric matrix $A\in \R^{2\times 2}$, where $|\cdot|_{\tr}$ and $|\cdot|_2$ denote the trace norm and the spectral norm, respectively.
Therefore, $(\rho,\bm, E)$ also constitutes a so-called dissipative weak solution.
We refer to the forthcoming preprint~\cite{Emil} for a rigorous proof of such a result.
\end{Rem}

In the Euler--Korteweg case $\mu=\eta=0$
the existence of energy-variational solutions in the sense of Definition~\ref{def:envar.ek} was shown in~\cite[Theorem 3.2]{EiterLasarzikSliwinski_envar} for any dimension $d\in\N$,
where $\Omega\subset\R^d$ was a bounded domain
and suitable boundary conditions were considered. 
This result was derived from an existence theory for a general class of evolution equations,
where energy-variational solutions were constructed through a time-incremental minimization scheme.
The proof can directly be adapted to the present case with periodic boundary conditions.
In this article, we follow a different direction
and show that the sequence of solutions $(\rho^h,\bu^h)$ 
to  the finite-volume scheme~\eqref{eq:EK.approx} converges
and that its limit is an energy-variational solution.

\begin{The}\label{thm:limit}
    Let $\kappa>0$ and $\mu,\eta\geq0$,
    and let $p(\rho)=k\rho^\gamma$ for some $k>0$ and $\gamma>1$.
    Let $(\rho_0,\bu_0)\colon\mathbb T^2\to\R\times\R^2$ be measurable such that $\calE(\rho_0,\rho_0\bu_0)<\infty$.
    Assume that for all $h>0$ 
    the solutions $(\rho^h,\bu^h)$ 
    to the approximate system~\eqref{eq:EK.approx}
    exist in a uniform time interval $(0,T)$ with $T>0$.
    Suppose that $\lambda$ chosen by~\eqref{eq:lambda} satisfies
    \begin{equation}
    \label{eq:lambda.limit}
        \lim_{h\to 0}\lambda h 
        =\lim_{h\to0} \frac{h}{\lambda}=0.
    \end{equation}
    Then there exists $(\rho,\bm,E)$ satisfying~\eqref{eq:fct.class}
    and a (not relabeled) subsequence of $(\rho^h,\bu^h,E^h)_h$
    such that
    \begin{subequations}
    \label{eq:conv.ek}
    \begin{align}
        \rho^h&\xrightharpoonup{*}\rho &&\text{in } L^\infty(0,T;L^\gamma(\mathbb T^2)),
        \label{eq:conv.rho.ws}
        \\        
        \nabla_h^+\rho^h, \nabla_h^-\rho, \nabla_h^c\rho&\rightharpoonup\nabla\rho &&\text{in } L^\infty(0,T;L^2(\mathbb T^2;\R^2)),
        \label{eq:conv.drho.w}
        \\
        \rho^h&\to\rho&&\text{in } C([0,T];L^p(\mathbb T^2)) \text{ for all }p\in[1,\infty),
        \label{eq:conv.rho.s}
        \\
        \rho^h \bu^h &\xrightharpoonup{*}\bm &&\text{in } L^\infty(0,T;L^{\frac{2\gamma}{\gamma+1}}(\mathbb T^2;\R^2)),
        \label{eq:conv.rhou.ws}
        \\
        E^h&\xrightharpoonup{*}E &&\text{in }\BV([0,T]),
        \label{eq:conv.E.ws}
    \end{align}
    \end{subequations}
    as $h\to 0$.
    If $\mu>0$ and $\eta>0$, then it holds $\bm=\rho \bu$ for some $\bu\in L^2(0,T;H^1(\mathbb T^2;\R^2))$ with
    \begin{subequations}
    \label{eq:conv.nsk}
    \begin{align}
         \bu^h&\rightharpoonup\bu &&\text{in } L^2(0,T; L^2(\mathbb T^2;\R^2)),
         \label{eq:conv.u.w}
         \\
         \nabla_h^+ \bu^h&\rightharpoonup \nabla \bu \qquad &&\text{in } L^2(0,T;L^2(\mathbb T^2;\R^{2\times 2})),
        \label{eq:conv.du.w}   
        \\
        \rho^h\bu^h &\to \bm &&\text{in } L^2(0,T;L^q(\mathbb T^2;\R^2)) \text{ for all }q\in[1,2),
        \label{eq:conv.m.s}
    \end{align}
    \end{subequations}
    as $h\to0$.
    In both cases, 
    $(\rho,\bm,E)$ is an energy-variational solution 
    in the sense of Definition~\ref{def:envar.ek}
    with $\mathcal K$ given by~\eqref{eq:K}.
\end{The}

\begin{Rem}
    As pointed out in Remark \ref{rem:lambda1}, the parameter $\lambda$ could also be chosen larger than in \eqref{eq:lambda}, i.e. $\lim_{h \rightarrow 0} \tfrac{h}{\lambda}=0$ is not really a restriction. The situation is more complex for $\lim_{h \rightarrow 0} {h}{\lambda}=0$. In our experience, in stable numerical simulations that do not specifically study vacuum, $\rho^h$ is uniformly-in-$h$ bounded from above and below and $|\bu^h|$ is uniformly-in-$h$ bounded from above, so that one can choose $\lambda$ uniform-in-$h$ and \eqref{eq:lambda.limit} is satisfied. However, strictly speaking, we cannot rule out that $\lambda$ goes to infinity for $h \rightarrow 0$.
\end{Rem}

\begin{Rem}\label{thm:better.conv}
    If $\mu,\eta>0$ and if the 
    density remains uniformly bounded from below,
    we can further improve the convergence
    and obtain strong convergence of the velocity.
    More precisely, let $\underline\rho>0$ such that $\rho^h\geq\underline\rho$ for all $h>0$.
    Then it holds
    \begin{equation}
        \bu^h \to \bu \quad\text{in } L^2(0,T;L^q(\mathbb T^2;\R^2)) \text{ for all }q\in[1,2).
        \label{eq:conv.u.s}
    \end{equation}
    Indeed, since
    \[
    \bu^h-\bu=\frac{\bm^h-\bm}{\rho^h}+\bm\Big(\frac{1}{\rho^h}-\frac{1}{\rho}\Big),
    \]
    the convergence~\eqref{eq:conv.u.s} follows from~\eqref{eq:conv.rho.s}, \eqref{eq:conv.m.s} and dominated convergence.
\end{Rem}

\begin{Rem}[Convergence of the initial values]
The  approximate initial values $ \rho_0^h$ and $\bm_0^h$ in~\eqref{eq:definitial} are defined via local averages  through~\eqref{eq:initialdata.cellwise}.  As the energy is finite in the initial values, $ \mathcal E(\rho_0,\bm_0) < \infty $, we know that $ \rho_0 \in L^\gamma(\mathbb T^2) $ as well as $ \bm_0 \in L^{\frac{2\gamma}{\gamma+1}}(\mathbb T^2)$. Standard theory implies the strong convergence
$$
\begin{aligned}
    \rho^h_0&\to\rho_0 &&\text{in } L^\gamma(\mathbb T^2),
        \quad&& \quad       
        \bm_0^h\to \bm_0 &&\text{in } L^{\frac{2\gamma}{\gamma+1}}(\mathbb T^2;\R^2),
        \quad && \quad
        \nabla^+_h\rho_0^h \to  \nabla \rho_0&&\text{in } L^2(\mathbb T^2;\R^2)
\end{aligned}
$$
    as $h\to 0$. Moreover, for $ \rho_0 (x) > \underline{\rho}>0 $ for a.e.~$x\in \mathbb T^2$, we observe from the convexity of the mapping  $ (\rho ,\bm) \mapsto \tilde{\eta}(\rho, \bm) + \frac{\kappa}{2}|\nabla^+_h \rho|^2$  and Jensen's inequality that 
    $$   \tilde{\eta}(\rho_0^h(x),\bm_0^h(x))  \leq \tilde{\eta}(\rho_0(x), \bm_0(x)) \,.$$ Due to $ \mathcal E(\rho_0,\bm_0) < \infty $, the right-hand side is integrable,
    and by Lebesgue's dominated convergence theorem, we observe that 
    $$
    \lim_{h\searrow 0} \int_{\mathbb T^2} \tilde{\eta}(\rho_0^h(x),\bm_0^h(x)) \dd x = \int_{\mathbb T^2} \tilde{\eta}(\rho_0(x),\bm_0(x)) \dd x \,.
    $$
    In conclusion, the energy-variational solutions constructed in Theorem~\ref{thm:limit} additionally fulfill $ E(0)= \mathcal E(\rho_0 , \bm_0) = \lim_{h\searrow0}E^h(0) $. 
\end{Rem}

To prove Theorem~\ref{thm:limit}, 
we begin by the derivation of suitable \textit{a priori} estimates.
The energy-dissipation inequality~\eqref{eq:edi.EK.disc}
directly implies uniform bounds for $(\rho^h,\bu^h)$,
which lead to a weakly convergent subsequence.
To obtain the asserted strong convergence properties for $\rho^h$, 
we derive additional estimates of the time derivative $\partial_t\rho^h$. 
An analogous estimate for $\partial_t\bm^h$ is derived later.
To shorten notation, we write 
\[
    \bbD_h^+(\bu^h):=\frac{1}{2}\big(\nabla_h^+ \bu^h + (\nabla_h^+ \bu^h)^T\big)
    \qquad
    \dv^+_h\bu^h:=\Dhpx u^h+\Dhpy v^h.
\]

\begin{Lem}\label{lem:bounds}
   Under the assumptions of Theorem \ref{thm:limit}, there exists a constant $C>0$,
   independent of $h$, $\kappa$, $\mu$ and $\eta$, such that
    \begin{equation}
    \label{eq:bounds.direct}
    \begin{aligned}
        &\normn{\rho^h |\bu^h|^2}_{L^\infty(0,T;L^1(\mathbb T^2))}
        +\normn{P(\rho^h)}_{L^\infty(0,T;L^1(\mathbb T^2))}
        +\sqrt{\kappa}\normn{\nabla^+_h \rho^h}_{L^\infty(0,T;L^2(\mathbb T^2))}
        \\
        &+\sqrt{2\mu}\norml{\dev\big(\bbD_h^+(\bu^h)\big) }_{L^2(0,T;L^2(\mathbb T^2))}
        +\sqrt{\eta}\normn{\dv^+_h\bu^h}_{L^2(0,T;L^2(\mathbb T^2))}
        \\
        &+\sqrt{\lambda h}\normn{\nabla^+_h \rho^h }_{L^2(0,T;L^2(\mathbb T^2))}
        +\sqrt{\kappa\lambda h}\normn{D_h^2\rho^h}_{L^2(0,T;L^2(\mathbb T^2))}
        +\normn{E^h}_{\BV([0,T])}\leq C.
    \end{aligned}
    \end{equation}
    Moreover, 
    we have
    \begin{align}
        \label{eq:bounds.rhogamma}
        \normn{\rho^h}_{L^\infty(0,T;L^\gamma(\mathbb T^2))}
        +\normn{\rho^h u^h}_{L^\infty(0,T;L^\frac{2\gamma}{\gamma+1}(\mathbb T^2))}
        &\leq C,
   \\
        \label{eq:bound.tdrho}
        \norml{\partial_t\rho^h}_{L^\infty(0,T;W^{-1,\frac{2\gamma}{\gamma+1}}(\mathbb T^2))+ L^2(0,T;H^{-1}(\mathbb T^2))}&\leq C(1+\sqrt{\lambda h}).
    \end{align}
\end{Lem}

\begin{proof}
From~\eqref{eq:edi.EK.disc}
we directly have
\begin{equation}
\label{eq:edi.EK.approx}
\begin{aligned}
&\int_{\mathbb T^2} \frac{1}{2}\rho^h(t) |\bu^h(t)|^2 + P(\rho^h(t))+\frac{\kappa}{2}|\nabla^+_h\rho^h(t)|^2 \dd x
\\
&
+\int_0^t\int_{\mathbb T^2} 2\mu \big|\dev\big(\bbD_h^+(\bu^h)\big)\big|^2
+\eta\big|\dv_h^+\bu^h\big|^2
+\kappa \lambda h |D^2_h\rho^h|^2 \dd x \dd \tau
\\
&\qquad
\leq
\int_{\mathbb T^2} \frac{|\bm^h_0|^2}{2\rho^h_0}  + P(\rho^h_0)+\frac{\kappa}{2}|\nabla_h^+\rho^h_0|^2 \dd x.
\end{aligned}
\end{equation}
Observe that the integral on the right-hand side is uniformly bounded due to $\calE(\rho_0,\rho_0\bu_0)<\infty$ and Jensen's inequality.
Therefore,~\eqref{eq:edi.EK.approx}
and the monotonicity of $E^h$, which follows directly from~\eqref{eq:edi.EK.disc}, 
give~\eqref{eq:bounds.direct}.

Since $p(\rho)=k\rho^\gamma$ for some $k>0$, $\gamma>1$,
it holds $P(\rho)=k(\gamma-1)^{-1}\rho^\gamma$.
As $\rho^h u^h = \sqrt{\rho^h} (\sqrt{\rho^h}u^h)$,
we thus conclude~\eqref{eq:bounds.rhogamma} from H\"older's inequality and~\eqref{eq:bounds.direct}.
To obtain a bound for $\partial_t\rho^h$,
consider $\phi\in C(0,T;C^1({\mathbb T^2}))$.
By~\eqref{eq:EK.approx.cont} and discrete integration by parts for difference quotients, we have
\[
\begin{aligned}
&\int_0^T\int_{\mathbb T^2}\partial_t\rho^h\phi\dd x \dd t
=\int_0^T\int_{\mathbb T^2} 
-\Dhcx(\rho^h u^h)\phi -\Dhcy(\rho^h v^h)\phi  +\lambda h D_h^2\rho^h\phi \dd x \dd t
\\
&\ =\int_0^T\int_{\mathbb T^2} 
\rho^h \bu^h \cdot\nabla_h^c\phi -\lambda h \nabla^+_h\rho^h \cdot\nabla^+_h\phi \dd x \dd t
\\
&\ \leq C\big(\norml{\rho^h \bu^h}_{L^\infty(0,T;L^{\frac{2\gamma}{\gamma+1}}({\mathbb T^2}))}\norml{\nabla\phi}_{L^1(0,T;L^\frac{2\gamma}{\gamma-1}({\mathbb T^2}))}+\lambda h\norml{\nabla_h^+\rho^h}_{L^2(0,T;L^2({\mathbb T^2}))}\norml{\nabla\phi}_{L^2(0,T;L^2({\mathbb T^2}))}\big)
\end{aligned}
\]
since
\[
\norml{\nabla^\pm_h\phi(t)}_{L^p({\mathbb T^2})}
+\norml{\nabla_h^c\phi(t)}_{L^p({\mathbb T^2})}
\leq C\norml{\nabla\phi(t)}_{L^p({\mathbb T^2})}
\]
for any $p\in[1,\infty]$.
Using~\eqref{eq:bounds.direct},
this shows
that 
\[
\partial_t\rho^h\in \big(L^1(0,T;W^{1,\frac{2\gamma}{\gamma-1}}({\mathbb T^2}))\cap L^2(0,T;H^1({\mathbb T^2}))\big)^*
\simeq
L^\infty(0,T;W^{-1,\frac{2\gamma}{\gamma+1}}({\mathbb T^2}))+ L^2(0,T;H^{-1}({\mathbb T^2}))
\]
with the uniform bound~\eqref{eq:bound.tdrho}.
\end{proof}

While the claimed weak convergence follows directly from the previous 
uniform bounds, 
we use 
the Aubin--Lions lemma to conclude strong convergence
by invoking the compactness of specific embeddings.
To use those in the spatially discretized setting,
we define the $| \cdot |_{\mathrm{dG},2}$-seminorm on 
\[
\drhos:=\big\{
f^h\in L^2({\mathbb T^2})\mid f^h|_{T} \text{ is constant for all } T\in\mathcal T_h
\big\},
\]
where $\mathcal T_h=\big\{\Omega_{\icj}\mid i=1,\dots,M,\,j=1,\dots,N\big\}$,
by
\[
    | f^h|_{\mathrm{dG},2}^2 := \sum_{i=1}^M\sum_{j=1}^N (f_{i,j}- f_{i-1,j})^2+ (f_{i,j}- f_{i,j-1})^2,
\]
where $f^h|_{\Omega_{\icj}}\equiv f_\icj\in\R$ for $\icj=1,\dots,N$.
Note that for the piecewise constant functions considered here,
this definition coincides with the norm defined in \cite[eqn.~(5.1)]{diPietroErn}.
However, as we consider periodic boundary conditions, this merely defines a seminorm
in the present context.

As we see below, $L^2$-bounds of the discrete gradients
directly lead to a bound of the $|\cdot|_{\mathrm{dG}}$-seminorm.
To derive a uniform estimate of the velocity if $\mu,\eta>0$, 
we further use the following discrete Korn--Poincar\'e inequality.

\begin{Lem}
\label{lem:KornPoincare}
Let $\{\mathcal{T}_h\}_{h \in \mathcal{H}}$ a family of meshes.
Let $M,K>0$ and $\gamma>1$.
Then there exists $C=C(M,K,\gamma)>0$, independent of $h$, such that 
\begin{equation}
\label{eq:KornPoincare}
\normn{\bv^h}_{L^2({\mathbb T^2};\R^2)}
+|\bv^h|_{\mathrm{dG},2}
\leq C\Big({\norml{\nabla_h^+\bv^h+(\nabla_h^+\bv^h)^T}_{L^2({\mathbb T^2};\R^{2\times 2})}+\int_{{\mathbb T^2}}r|\bv^h|\dd x}\Big)
\end{equation}
for all $\bv^h\in \dvs$, $h>0$ and 
 all measurable functions $r\colon\mathbb T^2\to\R$ with
\begin{equation}
\label{eq:properties.r.KP}
r\geq 0,
\qquad
\int_{\mathbb T^2}r\,\dd x \geq M,
\qquad
\int_{\mathbb T^2}r^\gamma \leq K.
\end{equation}
\end{Lem}

\begin{proof}
   It is clear that \eqref{eq:KornPoincare} holds for each $h$ separately, since both sides define norms on the finite-dimensional space $\dvs$.
   We need to prove that the constant $C$ can be chosen uniform in $h$ and $r$.
    This can be shown following the proof of~\cite[Theorem 11.23]{FeireslNovotny2017singularlimits}, where a similar result was derived for a Lipschitz domain $\Omega\subset\R^N$ with $N>2$.
   We only sketch the parts that differ slightly from~\cite[Theorem 11.23]{FeireslNovotny2017singularlimits}.
   For a proof by contradiction, let us assume that there exist a sequence $h_n$ such that $h_n \rightarrow h\in[0,\infty)$ and functions $\bv_n \in \mathbb{P}^0(\mathcal{T}_{h_n}, \mathbb{R}^2)$ 
   and $r_n\in L^\gamma(\mathbb T^2)$ such that
   \begin{align}
       \normn{\bv_n}_{L^2(\mathbb T^2;\R^2)}
        +|\bv_n|_{\mathrm{dG},2} &=1,\\
       {\norml{\nabla_h^+\bv_n+(\nabla_h^+\bv_n)^T}_{L^2(\mathbb T^2;\R^{2\times 2})}+\int_{\mathbb T^2}r_n|\bv_n|\dd x} &\leq \frac{1}{n},
   \end{align}
    and $r_n\rightharpoonup r$ in $L^\gamma(\mathbb T^2)$ for some $r$ satisfying~\eqref{eq:properties.r.KP}.
    Moreover, the sequence $\bv_n$ admits a weakly convergent subsequence with limit $\bv\in\mathbb{P}^0(\mathcal{T}_{h}, \mathbb{R}^2)$
    if $h>0$, and with $\bv \in H^1(\mathbb T^2; \mathbb{R}^2)$ due to \cite[Thm. 5.6]{diPietroErn}. 
    From this, we derive a contradiction as in the proof of \cite[Theorem 11.23]{FeireslNovotny2017singularlimits}.
\end{proof}

From the uniform bounds in Lemma~\ref{lem:bounds},
we conclude the existence of convergent subsequences
as claimed in Theorem~\ref{thm:limit}.

\begin{Lem}
    \label{lem:convergence}
    There exists a triple $(\rho,\bm,E)$ satisfying~\eqref{eq:fct.class}
    and a (not relabeled) subsequence of $(\rho^h,\bu^h,E^h)_h$
    that satisfy~\eqref{eq:conv.ek}.
    Moreover, if $\mu,\eta>0$, then there exists $\bu\in L^2(0,T;H^1(\mathbb T^2;\R^2))$ 
    with $\bm=\rho \bu$ 
    and such that~\eqref{eq:conv.u.w} and~\eqref{eq:conv.du.w} hold.
\end{Lem}
\begin{proof}
    The existence of elements $\rho$, $\bm$ and $E$ as in~\eqref{eq:conv.rho.ws},~\eqref{eq:conv.rhou.ws} and~\eqref{eq:conv.E.ws} follows directly from Lemma~\ref{lem:bounds} and the Banach--Alaoglu compactness theorem.
    In the same way, we conclude the existence of
    functions $\boldsymbol g^+$, $\boldsymbol g^-$, $\boldsymbol g^c$ such that
    \[
        \nabla_h^+\rho^h\rightharpoonup \boldsymbol g^+, \ \nabla_h^-\rho^h\rightharpoonup \boldsymbol g^-, \ 
        \nabla_h^c\rho^h\rightharpoonup \boldsymbol g^c \qquad\text{in } L^\infty(0,T;L^2(\mathbb T^2;\R^2)),
    \]
    as $h\to0$.
    One readily sees that $\boldsymbol g^+=\boldsymbol g^-=\boldsymbol g^c=\nabla\rho$
    in the distributional sense,
    which yields~\eqref{eq:conv.drho.w}.
    Similarly,~\eqref{eq:bound.tdrho} implies 
    \begin{equation}\label{eq:conv.dtrho.ws}
    \partial_t\rho^h\xrightharpoonup{*}\partial_t\rho \quad \text{in } 
    L^\infty(0,T;W^{-1,\frac{2\gamma}{\gamma+1}}(\mathbb T^2))+ L^2(0,T;H^{-1}(\mathbb T^2)).
    \end{equation}
    We further have
    \[
    \begin{aligned}
    | \rho^h(t)|_{\mathrm{dG},2}^2 
    &=\sum_{i,j=1}^N h_x^2(\fx \rho_\icj(t))^2+ h_y^2(\fy \rho_\icj(t))^2
    \\
    &\leq C\sum_{i,j=1}^N h_x h_y \big((\fx \rho_\icj(t))^2+(\fy \rho_\icj(t))^2\big)
    \leq C\normn{\nabla^+_h \rho^h(t)}_{L^\infty(0,T;L^2(\mathbb T^2))}^2,
    \end{aligned}
    \] 
    for a.a.~$t\in(0,T)$, where the left-hand side is uniformly bounded.
    Combining the discrete Rellich--Kondrachov theorem from~\cite[Theorem 5.6]{diPietroErn}
    with~\eqref{eq:conv.dtrho.ws},
    we conclude the strong convergence~\eqref{eq:conv.rho.s}
    from the Aubin--Lions lemma.    
    
    If $\mu>0$, we argue as before to see
    \[
    \normn{\bu^h}_{L^2(0,T;L^2(\mathbb T^2;\R^2))}^2
    +\normn{\nabla^+_h\bu^h}_{L^2(0,T;L^2(\mathbb T^2;\R^{2\times 2}))}^2
    \leq
    \normn{\bu^h}_{L^2(0,T;L^2(\mathbb T^2;\R^2))}^2
    +C\int_0^T|\bu^h|_{\mathrm{dG},2}^2 \dd t.
    \]
    Since the right-hand side is uniformly bounded due to Lemma~\ref{lem:KornPoincare}
    and Lemma~\ref{lem:bounds},
    we conclude the existence of $\bu$ with~\eqref{eq:conv.u.w} and~\eqref{eq:conv.du.w}.
    Moreover, a combination of the strong convergence~\eqref{eq:conv.rho.s} and the weak convergence~\eqref{eq:conv.u.w} yields
    $\rho^h\bu^h\rightharpoonup\rho\bu$ in $L^1(0,T;L^q(\mathbb T^2))$
    for all $q\in[1,2)$,
    so that $\rho\bu=\bm$ in virtue of~\eqref{eq:conv.rhou.ws}.
\end{proof}

To complete the proof of Theorem~\ref{thm:limit},
it remains to pass to the limit 
in the discretized formulation~\eqref{eq:edi.EK.disc} in a suitable way.
To this end, we invoke 
the regularity weight~$\calK$ from~\eqref{eq:K}.
Observe that the strong convergence~\eqref{eq:conv.m.s} is not necessary for this limit passage,
and we postpone its verification to the end of the proof.

\begin{proof}[Proof of Theorem~\ref{thm:limit}]
First of all, the function $E^h$ is nonincreasing for all $h>0$
by~\eqref{eq:edi.EK.disc}.
By the weak* convergence from~\eqref{eq:conv.E.ws} and Helly's selection theorem,
we deduce that $E$ satisfies~\eqref{eq:envar.ek.enineq}.
By weak lower semicontinuity, we also infer $\calE(\rho,m)\leq E$ a.e.
from the trivial inequality $\int_{\mathbb T^2} e^h(t)\dd x\leq E^h(t)$.

To obtain~\eqref{eq:envar.ek.cont},
we multiply the formulation~\eqref{eq:EK.approx.cont} with $\psi\in C^1([0,T); C^1(\mathbb T^2))$
and integrate in space and time.
Using (discrete) integration by parts, this leads to
\begin{equation}
	\int_0^T\int_{\mathbb T^2}-\rho^h\partial_t\psi - (\rho^h \bu^h)\cdot\nabla_h^c\psi +\lambda h \nabla_h^+\rho^h\cdot\nabla_h^+\psi\dd x \dd t
    =\int_{\mathbb T^2}\rho_0^h\psi(0)\dd x .
    \label{eq:cont.disc.weak}
\end{equation}
Note that the Cauchy--Schwarz inequality and~\eqref{eq:bounds.direct}  yield
\[
\begin{aligned}
\int_0^T\int_{\mathbb T^2}\lambda h \nabla_h^+\rho^h\cdot\nabla_h^+\psi\dd x \dd t
&\leq \lambda h \normn{\nabla_h^+\rho^h}_{L^2(0,T;L^2(\mathbb T^2))}\normn{\nabla_h^+\psi}_{L^2(0,T;L^2(\mathbb T^2))}
\\
&\leq \sqrt{\lambda h}\normn{\nabla_h^+\psi}_{L^2(0,T;L^2(\mathbb T^2))}\to 0
\end{aligned}
\]
as $h\to0$.
In virtue of the weak* convergence from~\eqref{eq:conv.rho.ws} and~\eqref{eq:conv.rhou.ws},
we thus obtain~\eqref{eq:envar.ek.cont} by passing to the limit $h\to0$ in~\eqref{eq:cont.disc.weak}.

To obtain~\eqref{eq:envar.ek.mom}, let $\bvarphi=(\varphi_x,\varphi_y)\in C^1([0,T);C^2(\mathbb T^2;\R^2))$.
We multiply the formulation~\eqref{eq:EK.approx.mom.u} with $\varphi_x$ 
and~\eqref{eq:EK.approx.mom.v} with $\varphi_y$
and add both identities.
Proceeding as before, we obtain
\begin{equation}
    \begin{split}
	&\int_0^T\int_{\mathbb T^2} -\rho^h \bu^h\cdot\partial_t\bvarphi -(\rho^h \bu^h\otimes\bu^h+p(\rho^h)\bbI ):\nabla_h^c\bvarphi 
    -\lambda h  \rho^h \bu^h\cdot D^2_h\bvarphi 
    \dd x\dd t 
    \\
    &\quad
    + \int_0^T\int_{\mathbb T^2} 
    2\mu\dev\big(\bbD_h^+(\bu^h)\big):\nabla_h^+\bvarphi
+\eta(\dv_h^+\bu^h)\,(\dv_h^+\bvarphi)
    \dd x \dd t
    \\
    &\quad
    + \int_0^T\int_{\mathbb T^2} \kappa \Big[ 
    \frac{\rho^h D^2_h\rho^h(\cdot+h_x\be_x)+\rho^h(\cdot+h_x\be_x) D^2_h\rho^h}{2}\Dhpx\varphi_x
    - \demi (\Dhpx\rho^h)^2\Dhpx\varphi_x
    \\
    &\qquad\qquad\qquad
    +\frac{ \rho^h D^2_h\rho^h(\cdot+h_y\be_y)+\rho^h(\cdot+h_y\be_y) D^2_h\rho^h}{2}\Dhpy\varphi_y
    - \demi (\Dhpy\rho^h)^2\Dhpy\varphi_y
    \\
    &\qquad\qquad\qquad
    +\demi\Dhmy\rho^h\Dhmy \rho^h(\cdot+h_x\be_x)\Dhpx\varphi_x
    +\demi\Dhmx\rho^h\Dhmx \rho^h(\cdot+h_y\be_y)\Dhpy\varphi_y
    \\
    &\qquad\qquad\qquad
    -\Dhcx\rho^h\, \Dhpy \rho^h\Dhpy\varphi_x
    -\Dhcy\rho^h\, \Dhpx \rho^h\Dhpx\varphi_y
    \Big] \dd x \dd t
    \\
    &\qquad
    =\int_{\mathbb T^2} \rho_0^h \bu_0^h\cdot\bvarphi(0)\dd x.
    \end{split}
    \label{eq:mom.disc.weak.1}
\end{equation}
Combining discrete integration by parts with the discrete product rule
\[
D^\pm_{h,x} [fg]
=D^\pm_{h,x}f\,g(\cdot\pm h_x\be_x) + f\, D^\pm_{h,x}g,
\qquad
D^\pm_{h,y} [fg]
=D^\pm_{h,y}f\,g(\cdot\pm h_y\be_y) + f\, D^\pm_{h,y}g,
\]
we further have
\[
\begin{aligned}
&\int_{\mathbb T^2} \frac{\rho^h D^2_h\rho^h(\cdot+h_x\be_x)+\rho^h(\cdot+h_x\be_x) D^2_h\rho^h}{2} \Dhpx\varphi_x\dd x
\\
&\qquad
=-\frac{1}{2}\int_{\mathbb T^2}\Dhmx\rho^h\Dhmx\rho^h(\cdot+h_x\be_x)\Dhpx\varphi_x(\cdot-h_x\be_x)
+\rho^h\Dhmx\rho^h(\cdot+h_x\be_x)\Dhmx\Dhpx\varphi_x\dd x
\\
&\qquad\qquad
-\frac{1}{2}\int_{\mathbb T^2}\Dhmy\rho^h\Dhmy\rho^h(\cdot+h_x\be_x)\Dhpx\varphi_x(\cdot-h_y\be_y)
+\rho^h\Dhmy\rho^h(\cdot+h_x\be_x)\Dhmy\Dhpx\varphi_x\dd x
\\
&\qquad\qquad
-\frac{1}{2}\int_{\mathbb T^2} \Dhmx\rho^h(\cdot+h_x\be_x) \Dhmx\rho^h \Dhpx\varphi_x
+\rho^h\Dhmx\rho^h\Dhmx\Dhpx\varphi_x\dd x
\\
&\qquad\qquad
-\frac{1}{2}\int_{\mathbb T^2} 
\Dhmy\rho^h(\cdot+h_x\be_x) \Dhmy\rho^h \Dhpx\varphi_x(\cdot-h_y\be_y)
+ \rho^h(\cdot+h_x\be_x)\Dhmy\rho^h\Dhmy\Dhpx\varphi_x\dd x
\\
&\qquad
=-\int_{\mathbb T^2} \Dhpx\rho^h \Dhmx\rho^h \Dhcx\varphi_x
+ \rho^h\Dhcx\rho^h\Dhmx\Dhpx\varphi_x\dd x
\\
&\qquad\qquad
-\int_{\mathbb T^2}
\Dhmy\rho^h \Dhmy\rho^h(\cdot+h_x\be_x)\Dhpx\varphi_x(\cdot-h_y\be_y)\dd x
\\
&\qquad\qquad
-\int_{\mathbb T^2}\frac{\rho^h\Dhmy\rho^h(\cdot+h_x\be_x)+\rho^h(\cdot+h_x\be_x)\Dhmy\rho^h}{2}\Dhmy\Dhpx\varphi_x \dd x,
\end{aligned}
\]
and analogously,
\[
\begin{aligned}
&\int_{\mathbb T^2} \frac{\rho^h D^2_h\rho^h(\cdot+h_y\be_y)+\rho^h(\cdot+h_y\be_y) D^2_h\rho^h}{2} \Dhpy\varphi_y\dd x
\\
&\qquad
=-\int_{\mathbb T^2} \Dhpy\rho^h \Dhmy\rho^h \Dhcy\varphi_y
+ \rho^h\Dhcy\rho^h\Dhmy\Dhpy\varphi_y\dd x
\\
&\qquad\qquad
-\int_{\mathbb T^2}\Dhmx\rho^h \Dhmx\rho^h(\cdot+h_y\be_y) \Dhpy\varphi_y(\cdot-h_x\be_x)\dd x
\\
&\qquad\qquad
-\int_{\mathbb T^2}\frac{\rho^h\Dhmx\rho^h(\cdot+h_y\be_y)+\rho^h(\cdot+h_y\be_y)\Dhmx\rho^h}{2}\Dhmx\Dhpy\varphi_y \dd x.
\end{aligned}
\]
If we write $\bm^h=\rho^h\bu^h$,
and use the identities
\[
\begin{aligned}
\frac{\rho^h\Dhmx\rho^h(\cdot+h_y\be_y)+\rho^h(\cdot+h_y\be_y)\Dhmx\rho^h}{2}
&=\rho^h\Dhmx\rho^h+\frac{h_y}{2}  \rho^h \Dhpy\Dhmx\rho^h+\frac{h_y}{2}\Dhmx\rho^h\Dhpy\rho^h,
\\
\frac{\rho^h\Dhmy\rho^h(\cdot+h_x\be_x)+\rho^h(\cdot+h_x\be_x)\Dhmy\rho^h}{2}
&=\rho^h\Dhmy\rho^h+\frac{h_x}{2}  \rho^h \Dhpx\Dhmy\rho^h+\frac{h_x}{2}\Dhmy\rho^h\Dhpx\rho^h,
\end{aligned}
\]
this allows us to reformulate~\eqref{eq:mom.disc.weak.1}
as
\begin{equation}
    \begin{split}
    &\int_0^T\int_{\mathbb T^2} \bm^h\cdot\partial_t\bvarphi +\Big(\frac{\bm^h\otimes\bm^h}{\rho^h}+p(\rho^h)\bbI \Big) :\nabla_h^c\bvarphi
    +\lambda h \bm^h\cdot D^2_h\bvarphi \dd x\dd t 
    \\
    &
    + \int_0^T\int_{\mathbb T^2} 
    -2\mu\dev\big(\bbD_h^+(\bu^h)\big):\nabla_h^+\bvarphi
    -\eta(\dv_h^+\bu^h)\,(\dv_h^+\bvarphi)
    \dd x \dd t
    \\
	  &
    + \int_0^T\int_{\mathbb T^2} \kappa \Big[ 
    \big[\rho^h\Dhmx\rho^h+\frac{h_y}{2}  \rho^h \Dhpy\Dhmx\rho^h+\frac{h_y}{2}\Dhmx\rho^h\Dhpy\rho^h\big]\Dhmy\Dhpx\varphi_x
    \\
    &\qquad\qquad\qquad
    +\big[\rho^h\Dhmy\rho^h+\frac{h_x}{2}  \rho^h \Dhpx\Dhmy\rho^h+\frac{h_x}{2}\Dhmy\rho^h\Dhpx\rho^h\big]\Dhmx\Dhpy\varphi_y
    \\
    &\qquad\qquad\qquad
    + \rho^h\Dhcx\rho^h\Dhmx\Dhpx\varphi_x+ \rho^h\Dhcy\rho^h\Dhmy\Dhpy\varphi_y
    \\
    &\qquad\qquad\qquad
    +\demi (\Dhpx\rho^h)^2 \Dhpx\varphi_x
    + \demi (\Dhpy\rho^h)^2 \Dhpy\varphi_y
    \\
    &\qquad\qquad\qquad
    +\Dhmy\rho^h \Dhmy\rho^h(\cdot+h_x\be_x)\big[\Dhpx\varphi_x(\cdot-h_y\be_y)-\demi\Dhpx\varphi_x\big]
    \\
    &\qquad\qquad\qquad
    +\Dhmx\rho^h \Dhmx\rho^h(\cdot+h_y\be_y) \big[\Dhpy\varphi_y(\cdot-h_x\be_x)-{\demi\Dhpy\varphi_y }\big]
    \\
    &\qquad\qquad\qquad
    +\Dhpx\rho^h \Dhmx\rho^h \Dhcx\varphi_x
    +\Dhpy\rho^h \Dhmy\rho^h \Dhcy\varphi_y
    \\
    &\qquad\qquad\qquad
    +\Dhcx\rho^h\, \Dhpy \rho^h\Dhpy\varphi_x
    +\Dhcy\rho^h\, \Dhpx \rho^h\Dhpx\varphi_y
    \Big]  
    + \calK(\bvarphi)e^h\dd x\dd t
    \\
    &\qquad
    =-\int_{\mathbb T^2} \bm_0^h\cdot\bvarphi(0)\dd x
    +\int_0^T \calK(\bvarphi)E^h\dd t,
    \end{split}
    \label{eq:mom.disc.weak.2}
\end{equation}
where we multiplied with $-1$ and added the term $\calK(\bvarphi)E^h$ on both sides,
where $\calK$ is defined in~\eqref{eq:K}.

To pass to the limit in~\eqref{eq:mom.disc.weak.2} in a suitable way, 
we first consider the linear terms on the left-hand side of~\eqref{eq:mom.disc.weak.2}.
By~\eqref{eq:conv.rhou.ws} and the bounds from~\eqref{eq:bounds.direct},
we can argue as for the continuity equation to conclude
\[
\lim_{h\to0}
\int_0^T\int_{\mathbb T^2} \bm^h\cdot\partial_t\bvarphi +\lambda h \bm^h\cdot D^2_h\bvarphi  \dd x\dd t 
=\int_0^T\int_{\mathbb T^2} \bm\cdot\partial_t\bvarphi \dd x\dd t .
\]
In particular, here we use that 
\[
\int_0^T\int_{\mathbb T^2}\lambda h \bm^h\cdot D^2_h\bvarphi\dd x \dd t
\leq \lambda h \normn{\bm^h}_{L^\infty(0,T;L^\frac{2\gamma}{\gamma+1}({\mathbb T^2}))}\normn{D^2_h\bvarphi}_{L^1(0,T;L^\frac{2\gamma}{\gamma-1}({\mathbb T^2}))}
\to 0
\]
as $h\to0$ due to~\eqref{eq:bounds.rhogamma}
and~\eqref{eq:lambda.limit}.
For $\mu,\eta>0$, by linearity of the viscous terms
and the weak convergence from~\eqref{eq:conv.du.w}, we further have
\[
\begin{aligned}
\lim_{h\to0}
&\int_0^T\int_{\mathbb T^2} -2\mu\dev\big(\bbD_h^+(\bu^h)\big):\nabla_h^+\bvarphi
    -\eta(\dv_h^+\bu^h)\,(\dv_h^+\bvarphi)\dd x\dd t 
\\
&=\int_0^T\int_{\mathbb T^2} -\bbS(\bu):\nabla\bvarphi\dd x\dd t 
=\int_0^T\int_{\mathbb T^2} -\bbS(\bu):\bbD(\bvarphi)\dd x\dd t .
\end{aligned}
\]
Moreover, we use~\eqref{eq:conv.E.ws} to pass to the limit on the right-hand side of~\eqref{eq:mom.disc.weak.2}, 
which yields
\[
\lim_{h\to0}\int_{\mathbb T^2} - \bm_0^h\cdot\bvarphi(0)\dd x+\int_0^T \calK(\bvarphi)E^h\dd t
=\int_{\mathbb T^2} - \bm_0\cdot\bvarphi(0)\dd x +\int_0^T \calK(\bvarphi)E\dd t.
\]
Concerning the nonlinear terms, we first consider those with second-order difference quotients of $\bvarphi$. 
Combining the strong convergence~\eqref{eq:conv.rho.s} 
with the weak convergence~\eqref{eq:conv.drho.w},
we obtain
\[
\begin{aligned}
\lim_{h\to\infty}&\int_0^T\int_{\mathbb T^2} \kappa \Big[ 
    \big[\rho^h\Dhmx\rho^h+\frac{h_y}{2}  \rho^h \Dhpy\Dhmx\rho^h+\frac{h_y}{2}\Dhmx\rho^h\Dhpy\rho^h\big]\Dhmy\Dhpx\varphi_x
    \\
    &\qquad\qquad
    +\big[\rho^h\Dhmy\rho^h+\frac{h_x}{2}  \rho^h \Dhpx\Dhmy\rho^h+\frac{h_x}{2}\Dhmy\rho^h\Dhpx\rho^h\big]\Dhmx\Dhpy\varphi_y
    \\
    &\qquad\qquad
    + \rho^h\Dhcx\rho^h\Dhmx\Dhpx\varphi_x+ \rho^h\Dhcy\rho^h\Dhmy\Dhpy\varphi_y
    \Big] \dd x \dd t
    \\
    &=\int_0^T\int_{\mathbb T^2} \kappa \rho\nabla\rho\cdot\nabla\dv\bvarphi \dd x \dd t
\end{aligned}
\]
since we have $\normn{\nabla^+_h\rho^h}_{L^2(0,T;L^2({\mathbb T^2}))} \leq C$ and
\begin{equation}\label{eq:limit.D^2h.rho}
h\normn{D^2_h\rho^h}_{L^2(0,T;L^2({\mathbb T^2}))}
\leq \frac{C h}{\sqrt{\lambda h}}\to 0
\end{equation}
as $h\to 0$
by the uniform bounds from~\eqref{eq:bounds.direct}
and the assumption~\eqref{eq:lambda.limit}.
To treat the remaining nonlinear terms, we recall~\eqref{eq:energy.discrete}
and combine them with 
certain contributions from the term $\mathcal K(\bvarphi)e^h$.
For the convection term, we observe that
\[
\begin{aligned}
&\frac{\bm^h\otimes\bm^h}{\rho^h}:\nabla_h^c\bvarphi
+2\|(\nabla \bvarphi)_{\mathrm{sym},-}\|_{L^\infty({\mathbb T^2})}\frac{|\bm^h|^2}{2\rho^h}
\\
&\qquad
=\frac{\bm^h\otimes\bm^h}{\rho^h}:\Big(\nabla\bvarphi+\|(\nabla \bvarphi)_{\mathrm{sym},-}\|_{L^\infty({\mathbb T^2})}\bbI \Big)
+\frac{\bm^h\otimes\bm^h}{\rho^h}:\Big(\nabla_h^c\bvarphi-\nabla\bvarphi\Big).
\end{aligned}
\]
The first term is convex in $(\rho^h,\bm^h)$ and nonnegative
since the term in parenthesis is a positive semidefinite matrix,
and the second term converges to $0$ as $h\to0$.
By the weak* convergence from~\eqref{eq:conv.rho.ws} and~\eqref{eq:conv.rhou.ws}, 
we can thus pass to the limit inferior in the associated integral functional
and obtain
\[
\begin{aligned}
\liminf_{h\to0}&\int_0^T\int_{\mathbb T^2}\frac{\bm^h\otimes\bm^h}{\rho^h}:\nabla_h^c\bvarphi
+2\|(\nabla \bvarphi)_{\mathrm{sym},-}\|_{L^\infty({\mathbb T^2})}\frac{|\bm^h|^2}{2\rho^h}
\dd x \dd t
\\
&\quad
\geq \int_0^T\int_{\mathbb T^2}\frac{\bm\otimes\bm}{\rho}:\nabla\bvarphi
+2\|(\nabla \bvarphi)_{\mathrm{sym},-}\|_{L^\infty({\mathbb T^2})}\frac{|\bm|^2}{2\rho}
\dd x \dd t.
\end{aligned}
\]
In the same way, we can use the identity $(\gamma-1)P(\rho)=p(\rho)$ to conclude
\[
\begin{aligned}
\liminf_{h\to0}&\int_0^T\int_{\mathbb T^2} p(\rho^h)\bbI  :\nabla_h^c\bvarphi
+\max\{\gamma-1,1\}\|(\dv \bvarphi)_{-}\|_{L^\infty({\mathbb T^2})}P(\rho^h)
\dd x \dd t
\\
&\quad
\geq \int_0^T\int_{\mathbb T^2} p(\rho)\dv\bvarphi
+\max\{\gamma-1,1\}\|(\dv\bvarphi)_{-}\|_{L^\infty({\mathbb T^2})}P(\rho)
\dd x \dd t.
\end{aligned}
\]
For the remaining $\kappa$-related terms,
we use identities of the form 
\[
\begin{aligned}
\Dhmy\rho^h\Dhmy \rho^h(\cdot+h_x\be_x)
&=|\Dhpy\rho^h|^2+\big[\Dhmy\rho^h-\Dhpy\rho^h]\Dhpy\rho^h
\\
&\quad+\Dhmy\rho^h\big[\Dhmy\rho^h(\cdot+h_x\be_x)-\Dhmy\rho^h+\Dhmy\rho^h-\Dhpy\rho^h\big]
\\
&=|\Dhpy\rho^h|^2
-h_y\Dhpy\Dhmy\rho^h\Dhpy\rho^h
\\
&\quad+h_x\Dhpx\Dhmy\rho^h \Dhmy\rho^h - h_y \Dhpy\Dhmy\rho^h \Dhmy\rho^h.
\end{aligned}
\]
The last three terms converge to $0$ strongly in $L^1(0,T;L^1({\mathbb T^2}))$
by H\"older's inequality and~\eqref{eq:limit.D^2h.rho}.
Invoking this and related identities, 
we can thus reduce 
our considerations to lower semicontinuity with respect to forward difference quotients $\Dhpx\rho^h$ and $\Dhpy\rho^h$
appearing in a quadratic way.
Since nonnegative quadratic functions are convex, 
we then conclude from the weak convergence~\eqref{eq:conv.drho.w} that
\[
\begin{aligned}
    \liminf_{h\to0}&\int_0^T\int_{\mathbb T^2} \kappa\Big[\demi (\Dhpx\rho^h)^2 \Dhpx\varphi_x
    + \demi (\Dhpy\rho^h)^2 \Dhpy\varphi_y
    \\
    &\qquad\qquad\quad
    +\Dhmy\rho^h \Dhmy\rho^h(\cdot+h_x\be_x)\big[\Dhpx\varphi_x(\cdot-h_y\be_y)-\demi\Dhpx\varphi_x\big]
    \\
    &\qquad\qquad\quad
    +\Dhmx\rho^h \Dhmx\rho^h(\cdot+h_y\be_y) \big[\Dhpy\varphi_y(\cdot-h_x\be_x)-\demi\Dhpy\varphi_y\big]\Big]
    \\
    &\qquad\qquad+\max\{\gamma-1,1\}\|(\dv\bvarphi)_{-}\|_{L^\infty({\mathbb T^2})}
    \frac{\kappa}{2}\big|\nabla_h^+\rho^h\big|^2 
    \dd x \dd t
    \\
    &\geq \int_0^T\int_{\mathbb T^2} \frac{\kappa}{2} |\nabla\rho|^2\dv\bvarphi
    +\max\{\gamma-1,1\}\|(\dv \bvarphi)_{-}\|_{L^\infty({\mathbb T^2})}
    \frac{\kappa}{2}\big|\nabla\rho\big|^2 
    \dd x \dd t
\end{aligned}
\]
and
\[
\begin{aligned}
    \liminf_{h\to0}&\int_0^T\int_{\mathbb T^2} \kappa \Big[\Dhpx\rho^h \Dhmx\rho^h \Dhcx\varphi_x
    +\Dhpy\rho^h \Dhmy\rho^h \Dhcy\varphi_y
    \\
    &\qquad\qquad\quad
    +\Dhcx\rho^h\, \Dhpy \rho^h\Dhpy\varphi_x
    +\Dhcy\rho^h\, \Dhpx \rho^h\Dhpx\varphi_y\Big]
    \\
    &\qquad\qquad+2\|(\nabla \bvarphi)_{\mathrm{sym},-}\|_{L^\infty({\mathbb T^2})}\frac{\kappa}{2}|\nabla^+_h\rho^h|^2
    \dd x \dd t
    \\
    &\geq \int_0^T\int_{\mathbb T^2} \kappa \nabla\rho\otimes\nabla\rho:\nabla\bvarphi
    +2\|(\nabla\bvarphi)_{\mathrm{sym},-}\|_{L^\infty({\mathbb T^2})}\frac{\kappa}{2}|\nabla\rho|^2
    \dd x \dd t.
\end{aligned}
\]
In summary, we can pass to the limit inferior on the left-hand side of~\eqref{eq:mom.disc.weak.2}
to conclude~\eqref{eq:envar.ek.mom}.
We have thus shown that the limit $(\rho,\bm,\E)$ is an energy-variational solution in the sense of Definition~\ref{def:envar.ek}.

To complete the proof, it remains to show the strong convernce~\eqref{eq:conv.m.s} if $\mu,\eta>0$.
From~\eqref{eq:EK.approx.mom.u} and \eqref{eq:EK.approx.mom.v}, 
we first derive an additional uniform bound for the time derivative of the momentum $\bm^h=\rho^h\bu^h$.
More precisely, consider the weak formulation~\eqref{eq:mom.disc.weak.2}
for a test function $\bvarphi\in C_c^1((0,T);C^2({\mathbb T^2}))$.
As the terms related with $\mathcal K$ cancel each other,
we obtain the estimate
\[
\begin{aligned}
    \int_0^T\int_{\mathbb T^2} \bm^h\cdot\partial_t\bvarphi \dd x \dd t
    &\leq C\Big(\norml{\rho^h|\bu^h|^2}_{L^1(0,T;L^1({\mathbb T^2}))}+
    \norml{p(\rho^h)}_{L^1((0,T)\times{\mathbb T^2})} \Big)\normn{\nabla_h^c\bvarphi}_{L^\infty((0,T)\times{\mathbb T^2})}
    \\
    &\qquad
    +C\lambda h \normn{\bm^h}_{L^\infty(0,T;L^{\frac{2\gamma}{\gamma+1}}({\mathbb T^2}))}\normn{D^2_h\bvarphi}_{L^1(0,T;L^{\frac{2\gamma}{\gamma-1}}({\mathbb T^2}))} 
    \\
    &\qquad
    + C(\mu+\eta)\normn{\bbD^+_h(\bu^h)}_{L^2((0,T)\times{\mathbb T^2})}\normn{\nabla^+_h\bvarphi}_{L^2((0,T)\times{\mathbb T^2})}
    \\
	  &\qquad
    + C\kappa \normn{\rho^h}_{L^2((0,T)\times\mathbb T^2)}
    \normn{\nabla^+_h\rho^h}_{L^\infty(0,T;L^2(\mathbb T^2))}
    \normn{D_h^2\bvarphi}_{L^2(0,T;L^\infty(\mathbb T^2))}
    \\
    &\qquad
    +C\kappa\normn{\nabla^+_h\rho^h}_{L^\infty(0,T;L^2(\mathbb T^2))}^2
    \normn{\nabla^+_h\bvarphi}_{L^1(0,T;L^\infty(\mathbb T^2))}.
\end{aligned}
\]
In virtue of the uniform bounds from Lemma~\ref{lem:bounds},
this shows that $\partial_t\bm^h$ satisfies
\begin{equation}\label{eq:bound.m.dt}
\normn{\partial_t\bm^h}_{L^1(0,T;(W^{1,\infty}(\mathbb T^2))^*)+L^p(0,T;(W^{2,\frac{2\gamma}{\gamma-1}}(\mathbb T^2))^*)}\leq C(1+\lambda h+\frac{h\sqrt{h}}{\sqrt{\lambda}})
\end{equation}
for any $p\in[1,\infty)$,
which gives a uniform bound due to~\eqref{eq:lambda.limit}.
Moreover, by discrete product rule and Sobolev embeddings,
we infer for any $q\in[1,2)$
that
\begin{equation}
\label{eq:bound.m.dx}
\begin{aligned}
&\normn{\nabla^+_h\bm^h}_{L^2(0,T;L^q(\mathbb T^2))}
\\
&\quad\leq \normn{\rho^h}_{L^\infty(0,T;L^p(\mathbb T^2))} \normn{\nabla^+_h\bu^h}_{L^2(0,T;L^2(\mathbb T^2))}
+\normn{\bu^h}_{L^2(0,T;L^p(\mathbb T^2))} \normn{\nabla^+_h\rho^h}_{L^\infty(0,T;L^2(\mathbb T^2))}
\\
&\quad\leq C\normn{\nabla^*_h\rho^h}_{L^\infty(0,T;L^2(\mathbb T^2))} \normn{\nabla^*_h\bu^h}_{L^2(0,T;L^2(\mathbb T^2))}
\end{aligned}
\end{equation}
where $\frac{1}{p}+\frac{1}{2}=\frac{1}{q}$.
Due to~\eqref{eq:bounds.direct}
and the Korn-type inequality~\eqref{eq:KornPoincare},
the right-hand side is uniformly bounded.
Here, we use that the discretization scheme preserves the total momentum 
by Theorem~\ref{2Dconservation}.
As above, we can now invoke  
the discrete Rellich--Kondrachov theorem
and the Aubin--Lions lemma,
and conclude the strong convergence of a subsequence such that~\eqref{eq:conv.m.s}
for all $q\in[1,2)$.
This completes the proof.
\end{proof}
    
\section{Acknowledgement}
Financial support by the German Research Foundation (DFG), within  the projects No. 525866748, No. 525877563, and No. 526018747 of the Priority Programme - SPP 2410 Hyperbolic Balance Laws in Fluid Mechanics: Complexity, Scales, Randomness (CoScaRa) is acknowledged.
T.~Eiter's research has further been funded by DFG through grant CRC 1114 ``Scaling Cascades in Complex Systems'', Project Number 235221301, Project YIP. J. Giesselmann's research has further been funded by DFG through grant CRC TRR 154, Project Number 239904186, Project C05.

\bibliographystyle{abbrvurl}
\bibliography{literature}

\end{document}